\documentclass[11pt]{amsart}
\usepackage{amsmath,amssymb,amsfonts,exscale}

\usepackage[curve,matrix,arrow,cmtip]{xy}
\usepackage{verbatim}

\makeatletter
\def\subsubsection{\@startsection{subsubsection}{3}%
  \z@{.5\linespacing\@plus.7\linespacing}{-.5em}%
  {\normalfont\bfseries}}
\makeatother

\makeatletter
\def\subsection{\@startsection{subsection}{2}%
  \z@{.5\linespacing\@plus.7\linespacing}{.7\linespacing}%
  {\normalfont\bfseries}}
\makeatother

\DeclareSymbolFont{bbold}{U}{bbold}{m}{n}
\DeclareSymbolFontAlphabet{\mathbbold}{bbold}

\newdir^{ (}{{}*!/-3pt/\dir^{(}}    
\newdir_{ (}{{}*!/-3pt/\dir_{(}}

\newtheorem*{theorem A}{Theorem A}
\newtheorem*{theorem B}{Theorem B}

\def\Spec{\mathop{\rm Spec}\nolimits}

\def\Mat{\mathop{\rm Mat}\nolimits}

\def\GL{\mathop{\rm GL}\nolimits}
\def\SL{\mathop{\rm SL}\nolimits}

\def\Bun{\mathop{\rm Bun}\nolimits}
\def\Rep{\mathop{\rm Rep}\nolimits}

\def\Maps{\mathop{\rm Maps}\nolimits}

\def\VinBun{\mathop{\rm VinBun}\nolimits}

\def\Vin{\mathop{\rm Vin}\nolimits}
\def\Kost{\mathop{\rm Kostant}\nolimits}

\def\Qellbar{\mathop{\overline{\BQ}_\ell}\nolimits}
\def\IC{\mathop{\rm IC}\nolimits}

\def\add{{\rm add}}

\def\gr{{\rm gr}}

\def\id{{\rm id}}

\newbox\starbox 
\setbox\starbox=\hbox{\vphantom{\underline{\ }}\rlap{\kern4pt\small?}\smash{\lower2pt\hbox{\Large$\square$}}}

\def\hatE{{\mathchoice
  {\hbox{\rlap{\smash{\kern1pt\lower1pt\hbox{$\widehat{\phantom{\hbox{$E$}}}$}}}$E$}}
  {\hbox{\rlap{\smash{\kern1pt\lower1pt\hbox{$\widehat{\phantom{\hbox{$E$}}}$}}}$E$}}
  {\widehat E}
  {\widehat E}}}

\def\hatW{\hbox{\rlap{\smash{\lower1pt\hbox{$\widehat{\phantom{\hbox{$W$}}}$}}}$W$}}

\def\tildeW{\hbox{\rlap{\smash{\lower1pt\hbox{$\widetilde{\phantom{\hbox{$W$}}}$}}}$W$}}

\newbox\checkWbox
\setbox\checkWbox\hbox{\rlap{\smash{\kern.8pt\lower4pt\hbox{\huge \v{}}}}$W$}

\def\barBun{{\overline{\Bun}}}

\def\D{{\rm D}}
\def\sl{{\mathfrak{sl}}}

\def\circV{{\mathchoice{\circVbig}{\circVbig}{\circVscript}{\circVscriptscript}}}
\def\circVbig{\hbox{\text{\it\r{V}}}}
\def\circVscript{\hbox{\scriptsize\text{\it\r{V}}}}
\def\circVscriptscript{\mbox{\tiny\text{\it\r{V}}}}
\def\circVprime{{\mathchoice{\circV\kern1.8pt{}^\prime}{\circV\kern1.8pt{}^\prime}
                            {\circVscript\kern1.3pt{}^\prime}{\circVscriptscript\kern1pt{}^\prime}}}

\def\circVpprime{{\mathchoice{\circV\kern1.8pt{}^{\prime \prime}}{\circV\kern1.8pt{}^{\prime \prime}}
                            {\circVscript\kern1.3pt{}^{\prime \prime}}{\circVscriptscript\kern1pt{}^{\prime \prime}}}}

\def\lambdach{\check\lambda}
\def\Lambdach{\check\Lambda}
\def\alphacheck{\check\alpha}
\def\thetacheck{\check\theta}
\def\betacheck{\check\beta}

\let\epsilon\varepsilon
\let\setminus\smallsetminus

\let\leq\leqslant
\let\geq\geqslant

\newtheorem{theorem}[subsubsection]{Theorem}

\newtheorem{corollary}[subsubsection]{Corollary}

\newtheorem{proposition-definition}[subsubsection]{Proposition-Definition}
\newtheorem{theorem-definition}[subsubsection]{Theorem-Definition}

\newtheorem{lemma}[subsubsection]{Lemma}

\newtheorem{proposition}[subsubsection]{Proposition}
\newtheorem{question}[subsubsection]{Question}
\newtheorem{remark}[subsubsection]{Remark}

\newcommand{\BA}{{\mathbb{A}}}

\newcommand{\BG}{{\mathbb{G}}}

\newcommand{\BQ}{{\mathbb{Q}}}

\newcommand{\BY}{{\mathbb{Y}}}
\newcommand{\BZ}{{\mathbb{Z}}}

\newcommand{\Fn}{{\mathfrak{n}}}

\newcommand{\Fs}{{\mathfrak{s}}}

\newcommand{\CF}{{\mathcal F}}

\newcommand{\CI}{{\mathcal I}}

\newcommand{\CK}{{\mathcal K}}

\newcommand{\CO}{{\mathcal O}}
\newcommand{\CP}{{\mathcal P}}

\newcommand{\CY}{{\mathcal Y}}

\newcommand{\ssec}{\subsection}
\newcommand{\sssec}{\subsubsection}

\def\longto{\longrightarrow}
\def\into{\hookrightarrow}

\def\longinto{\lhook\joinrel\longrightarrow}

\def\longonto{\ontoover{\ }}

\newbox\mybox
\def\arrover#1{\mathrel{
       \setbox\mybox=\hbox spread 1.4em
              {\hfil$\scriptstyle#1\vphantom{g}$\hfil}
       \vbox{\offinterlineskip\copy\mybox
             \hbox to\wd\mybox{\rightarrowfill}}}}
\def\larrover#1{\mathrel{
       \setbox\mybox=\hbox spread 1.4em{\hfil$\scriptstyle#1$\hfil}
       \vbox{\offinterlineskip\copy\mybox
             \hbox to\wd\mybox{\leftarrowfill}}}}

\def\ontoover#1{\mathrel{
       \setbox\mybox=\hbox spread 1.4em{\hfil$\scriptstyle#1$\hfil}
       \vbox{\offinterlineskip\copy\mybox
             \hbox to\wd\mybox{\rightarrowfill\hskip-2.8mm
                               $\rightarrow$}}}}
\def\leftontoover#1{\mathrel{
       \setbox\mybox=\hbox spread 1.4em{\hfil$\scriptstyle#1$\hfil}
       \vbox{\offinterlineskip\copy\mybox
             \hbox to\wd\mybox{$\leftarrow$\hskip-2.8mm
                               \leftarrowfill}}}}

\newbox\invlimsymbol
\setbox\invlimsymbol=\vtop{\hbox{\rm lim}\vskip-8pt
        \hbox{\hskip1pt$\scriptstyle\longleftarrow$}\vskip-1pt}

\newbox\dirlimsymbol
\setbox\dirlimsymbol=\vtop{\hbox{\rm lim}\vskip-8pt
        \hbox{\hskip1pt$\scriptstyle\longrightarrow$}\vskip-1pt}

\begin{document}

\title[Monodromy and Vinberg fusion for $\VinBun_G$]{Monodromy and Vinberg fusion for the principal degeneration of the space of $G$-bundles}

\author{Simon Schieder}
\thanks{Dept. of Mathematics, MIT, Cambridge, MA 02139, USA}
\address{Dept. of Mathematics, MIT, Cambridge, MA 02139, USA}

\maketitle

\begin{abstract}
We study the geometry and the singularities of the \textit{principal direction} of the Drinfeld-Lafforgue-Vinberg degeneration of the moduli space of $G$-bundles $\Bun_G$ for an arbitrary reductive group $G$, and their relationship to the Langlands dual group $\check{G}$ of $G$.

The article consists of two parts.
In the first and main part, we study the monodromy action on the nearby cycles sheaf along the principal degeneration of $\Bun_G$ and relate it to the Langlands dual group $\check G$. We describe the weight-monodromy filtration on the nearby cycles and generalize the results of \cite{Sch1} from the case $G=\SL_2$ to the case of an arbitrary reductive group $G$. Our description is given in terms of the combinatorics of the Langlands dual group $\check G$ and generalizations of the Picard-Lefschetz oscillators found in \cite{Sch1}.
Our proofs in the first part use certain \textit{local models} for the principal degeneration of $\Bun_G$ whose geometry is studied in the second part.

Our local models simultaneously provide two types of degenerations of the Zastava spaces; these degenerations are of very different nature, and together equip the Zastava spaces with the geometric analog of a Hopf algebra structure. The first degeneration corresponds to the usual Beilinson-Drinfeld fusion of divisors on the curve. The second degeneration is new and corresponds to what we call \textit{Vinberg fusion}: It is obtained not by degenerating divisors on the curve, but by degenerating the group $G$ via the Vinberg semigroup. Furthermore, on the level of cohomology the degeneration corresponding to the Vinberg fusion gives rise to an algebra structure, while the degeneration corresponding to the Beilinson-Drinfeld fusion gives rise to a coalgebra structure; the compatibility between the two degenerations yields the Hopf algebra axiom.

It is natural to conjecture that this Hopf algebra agrees with the universal enveloping algebra of the positive part of the Langlands dual Lie algebra $\check{\mathfrak{g}}$. The above procedure would then yield a novel and highly geometric way to pass to the Langlands dual side: Elements of $\check{\mathfrak{g}}$ are represented as cycles on the above moduli spaces, and the Lie bracket of two elements is obtained by deforming the cartesian product cycle along the Vinberg degeneration.

\end{abstract}

\newpage

\tableofcontents

\newpage

\section{Introduction}

\bigskip

\ssec{Context and overview}

Let $X$ be a smooth projective curve over an algebraically closed field $k$, let $G$ be a reductive group over $k$, and let $\Bun_G$ denote the moduli stack of $G$-bundles on $X$.
Drinfeld has constructed (unpublished) a canonical compactification $\barBun_G$ of $\Bun_G$ which is of relevance both to the classical and geometric Langlands program; the compactification $\barBun_G$ is singular, and its definition relies on the Vinberg semigroup $\Vin_G$ of $G$ introduced by Vinberg (\cite{V}).

\medskip

While Drinfeld's definition of the compactification $\barBun_G$ appeared only recently in \cite{Sch2}, certain smooth open substacks of $\barBun_G$ in the special case $G = \GL_n$ were already used by Drinfeld and by L. Lafforgue in their seminal work on the Langlands correspondence for function fields (\cite{Dr1}, \cite{Dr2}, \cite{Laf}). The compactification $\barBun_G$ is however already singular for $G=\SL_2$, and for various applications in the classical and geometric Langlands program it is necessary to understand its singularities.
The study of the singularities of $\barBun_G$ was begun in \cite{Sch1} in the case $G=\SL_2$, and, with a different focus, in \cite{Sch2} for an arbitrary reductive group $G$. These articles also introduce a minor modification of the space $\barBun_G$ which we refer to as the \textit{Drinfeld-Lafforgue-Vinberg degeneration} of $\Bun_G$ and denote by $\VinBun_G$; it can be viewed as a canonical multi-parameter degeneration of $\Bun_G$ over an affine space.

\medskip

The study of the singularities of $\barBun_G$ and $\VinBun_G$ in \cite{Sch1}, \cite{Sch2}, and the present article is originally motivated by the geometric Langlands program (\cite{G3}, \cite{G4}), but has also already found applications to the classical theory. As examples of applications we list the study of Drinfeld's and Gaitsgory's \textit{miraculous duality} and \textit{strange functional equations} in \cite{G2} and \cite{Sch2}; the geometric construction of the \textit{Bernstein asymptotics map} in \cite{Sch2} conjectured by Sakellaridis (\cite{Sak1}, \cite{Sak2}, and also \cite{BK}, \cite{CY}); and the geometric construction of Drinfeld's and Wang's \textit{strange bilinear form on the space of automorphic forms} in \cite{DrW} and \cite{W2}, using \cite{Sch1} and \cite{Sch2}, respectively. Finally, the \textit{Picard-Lefschetz oscillators} -- certain perverse sheaves found in \cite{Sch1} for $G=\SL_2$ and generalized in the present work to arbitrary reductive groups $G$ -- have recently also been shown to appear in other deformation-theoretic contexts, such as in the degeneration of Whittaker sheaves in the work of Campbell (\cite{C}).

\medskip

The work discussed in this article consists of two parts: A first and main part, and a second part which is logically independent from the first; both are concerned with the study of the geometry of the \textit{principal degeneration} of $\Bun_G$, a one-parameter subfamily of the multi-parameter family $\VinBun_G$.
The first part of the present work continues the study of the singularities of the space $\VinBun_G$ begun in the articles \cite{Sch1} and \cite{Sch2}, though it is independent of these articles. The main theorem of the first part determines the weight-monodromy filtration on the nearby cycles sheaf of the principal degeneration of $\Bun_G$, generalizing the main theorem of \cite{Sch1} from the case $G=\SL_2$ to the case of an arbitrary reductive group $G$. While this is not visible in the case $G=\SL_2$ treated in \cite{Sch1}, the answer for an arbitrary reductive group achieves the passage to the Langlands dual side: Our description is given in terms of the combinatorics of the Langlands dual group $\check G$ of $G$ and generalizations of the \textit{Picard-Lefschetz oscillators} found in \cite{Sch1}.
We refer the reader to \cite[Sec. 1.3--1.5]{Sch1} for further background on how these results are related to the miraculous duality and the geometric Langlands program.

\medskip

The proofs of the results of the first part utilize certain \textit{local models} for the principal degeneration; the geometry of these local models is studied in further detail in a separate section. The contribution of this separate section is the construction of a novel geometric operation on the Zastava spaces that we call \textit{Vinberg fusion} and which naturally complements the usual Beilinson-Drinfeld fusion.

\bigskip

\ssec{The principal degeneration of $\Bun_G$}

Before discussing our main results, we first need to introduce the basic geometric objects needed for its formulation.

\sssec{The Vinberg semigroup $\Vin_G$}
In \cite{V} Vinberg has defined and studied a canonical multi-parameter degeneration $\Vin_G \to \BA^r$ of an arbitrary reductive group $G$ of semisimple rank $r$, the \textit{Vinberg semigroup}. Its fibers away from all coordinate planes are isomorphic to the group $G$. Its fibers over the coordinate planes afford group-theoretic descriptions in terms of the parabolic subgroups of $G$. While the Vinberg semigroup is singular, it possesses a certain well-behaved open subvariety which is closely related to the wonderful compactification constructed by De Concini and Procesi in \cite{DCP}.

\medskip

\sssec{The Drinfeld-Lafforgue-Vinberg degeneration $\VinBun_G$}
As the Vinberg semigroup $\Vin_G$ comes equipped with a natural $G \times G$-action, we may form the mapping stack
$$\Maps(X, \Vin_G / G \times G)$$
parametrizing maps from the curve $X$ to the quotient $\Vin_G / G \times G$. The Drinfeld-Lafforgue-Vinberg degeneration $\VinBun_G$ from \cite{Sch2} is then obtained from this mapping stack by imposing certain non-degeneracy conditions. The natural map $\Vin_G \to \BA^r$ induces a natural map
$$\VinBun_G \ \longto \ \BA^r \, .$$
Completely analogously to how $\Vin_G$ forms a canonical multi-parameter degeneration of the group $G$, this map realizes $\VinBun_G$ as a canonical multi-parameter degeneration of $\Bun_G$. The compactification $\barBun_G$ mentioned above can be obtained from $\VinBun_G$ as the quotient by a maximal torus $T$ of $G$.

\medskip

\sssec{The case $G = \SL_2$ from \cite{Sch1}}
For $G = \SL_2$ the degeneration $\VinBun_G$ may be described in concrete terms as follows: It parametrizes triples $(E_1, E_2, \varphi)$ consisting of two $\SL_2$-bundles $E_1$, $E_2$ on the curve $X$ together with a morphism of the associated vector bundles $\varphi: E_1 \to E_2$ which is required to be not the zero map. The map
$$\VinBun_G \ \longto \ \BA^1$$
mentioned above is obtained by taking the determinant of the map $\varphi$.

\medskip

\sssec{The principal degeneration $\VinBun_G^{princ}$}
In the present article we will only be interested in the one-parameter degeneration
$$\VinBun_G^{princ} \ \longto \ \BA^1$$
of $\Bun_G$ obtained by restricting the family $\VinBun_G \to \BA^r$ to a general line in $\BA^r$ passing through the origin; for concreteness one may choose the line passing through the origin and the point $(1,\ldots,1) \in \BA^r$. We refer to this degeneration as the \textit{principal degeneration} of $\Bun_G$ as its special fiber $\VinBun_G^{princ}|_0$ is naturally related to the Borel subgroup $B$ of $G$.

\medskip

\sssec{Stratification of the special fiber $\VinBun_G^{princ}|_0$}
The special fiber $\VinBun_G^{princ}|_0$ is singular, and we will introduce a \textit{defect stratification} for it: One can associate to each point in the special fiber a certain effective divisor on the curve $X$ valued in the monoid of positive coweights $\Lambdach_G^{pos}$ which governs the singularity of the point in the moduli space $\VinBun_G^{princ}$. The degree of this divisor forms an element of $\Lambdach_G^{pos}$ as well, and we refer to it as the \textit{defect} of the point. The strata of the defect stratification of the special fiber $\VinBun_G^{princ}|_0$ are then defined as the loci where the defect remains constant. For the purpose of this introduction we will denote the stratum of the special fiber corresponding to a positive coweight $\check\theta \in \Lambdach_G^{pos}$ by $\sideset{_{\check\theta}}{_G^{princ}}{\VinBun}|_0$.

\bigskip

\ssec{Main theorem about nearby cycles}

The main theorem of the first part of this article describes the weight-monodromy filtration on the nearby cycles sheaf $\Psi^{princ}$ of the principal degeneration $\VinBun_G^{princ} \to \BA^1$. To sketch its formulation, let $\barBun_B$ denote Drinfeld's relative compactification of the map $\Bun_B \to \Bun_G$, defined in \cite{BG1}; furthermore, for a positive coweight $\check\theta \in \Lambdach_G^{pos}$ let $X^{\check\theta}$ denote the space of $\Lambdach_G^{pos}$-valued effective divisors on the curve $X$.
For each $\check\theta \in \Lambdach_G^{pos}$ we then construct a surjective and finite map onto the strata closure $\overline{\sideset{_{\check\theta}}{_G^{princ}}{\VinBun}|_0}$ of the form
$$\bar f_{\check\theta}: \ \ \barBun_{B^-} \ \underset{\Bun_T}{\times} \ \bigl( \, X^{\check\theta} \, \times \barBun_B \bigr) \ \ \longonto \ \ \overline{\sideset{_{\check\theta}}{_G^{princ}}{\VinBun}|_0}$$
which restricts to an isomorphism on the interiors.
While the space $\barBun_B$ is itself singular, its IC-sheaf $\IC_{\barBun_B}$ is well-understood in Langlands-dual terms (\cite{FFKM}, \cite{BFGM}, \cite{BG2}). 
Broadly speaking, our main theorem about the nearby cycles of the principal degeneration $\VinBun_G^{princ} \to \BA^1$ asserts:

\medskip

\begin{theorem A}
The associated graded with respect to the weight-monodromy filtration on $\Psi^{princ}$ is equal to
$$\gr \, \Psi^{princ} \ \ \ \cong \ \ \ \bigoplus_{\check\theta \in \Lambdach_G^{pos}} \ \bar f_{\check\theta, *} \Bigl( \IC_{\barBun_{B^-}} \underset{ \, \Bun_T}{\boxtimes} \Bigl( \CF_{\check\theta} \ \boxtimes \ \IC_{\barBun_B} \Bigr) \Bigr)$$
as representations of the Lefschetz-$\sl_2$.
\end{theorem A}

\medskip

Here the $\CF_{\check\theta}$ denote certain novel perverse sheaves on the spaces of divisors $X^{\check\theta}$ which we will refer to as \textit{Picard-Lefschetz oscillators for $G$} and which govern the sheaf-theoretic description of the singularities of the principal degeneration $\VinBun_G^{princ}$. They form the correct generalization of the Picard-Lefschetz oscillators found for $G=\SL_2$ in \cite{Sch1}. In fact, for an arbitrary reductive group $G$, the $\CF_{\check\theta}$ combine various versions of the Picard-Lefschetz oscillators from \cite{Sch1} on the diagonals of $X^{\check\theta}$ in a combinatorial fashion that depends on the Langlands dual group $\check G$ of $G$. Like the Picard-Lefschetz oscillators from \cite{Sch1}, they by construction carry an action of the Lefschetz-$\sl_2$, and the above theorem asserts that the above isomorphism identifies this action with the monodromy action of the Lefschetz-$\sl_2$ on the associated graded $\gr \Psi^{princ}$.

\medskip

Unlike in the case $G = \SL_2$ studied in \cite{Sch1}, the perverse sheaf $\CF_{\thetacheck}$ is not equal to its intermediate extension from the ``disjoint locus'' of $X^{\thetacheck}$: Due to the existence of non-simple positive roots for the Langlands dual group~$\check G$ the perverse sheaf $\CF_{\thetacheck}$ possesses simple summands supported on the diagonals of $X^{\thetacheck}$. In particular, the proof strategy of \cite{Sch1} fails for an arbitrary reductive group $G$. Instead, the proof of Theorem A above given in the present work needs to ``reconstruct'' the summands on the diagonals prescribed by the Langlands dual group $\check{G}$.

\bigskip

\ssec{Vinberg fusion for local models}

Our proof of Theorem A makes use of certain \textit{local models} for the principal degeneration of $\Bun_G$ which were introduced in \cite{Sch2}; we refer to Section \ref{Proofs I} below for their construction and main properties. The local models feature the same singularities as the principal degeneration, but allow for inductive arguments due to the presence of a factorization structure in the sense of Beilinson and Drinfeld (\cite{BD1}, \cite{BD2}). They are related to the principal degeneration of $\Bun_G$ in the exact same way as the Zastava spaces from \cite{FM}, \cite{FFKM}, and \cite{BFGM} are related to Drinfeld's relative compactification $\barBun_B$, and we exploit this interplay exactly as in \cite{BFGM} or \cite{BG2}.

\medskip

In the last section of this article, Section \ref{Vinberg fusion and a geometric Hopf algebra structure}, we discuss why the geometry of these local models may be of interest in geometric representation theory, independently from their use in the proof of Theorem A. Our local models combine two quite different types of degenerations of the Zastava spaces compatibly into one total space: They are naturally fibered over the parameter spaces $X^{\check\theta} \times \BA^1$.

\medskip

The degeneration corresponding to changing the divisor in $X^{\check\theta}$ yields the usual Beilinson-Drinfeld fusion operation for the Zastava spaces; it deforms a given Zastava space to a product of Zastava spaces, analogous to a coalgebra structure.

\medskip

The degeneration corresponding to the $\BA^1$-factor is new, and yields an operation that we call \textit{Vinberg fusion}. It is obtained not by degenerating the divisor but rather by degenerating the group via the Vinberg semigroup $\Vin_G$; it deforms a product of Zastava spaces to a single Zastava space, analogous to an algebra structure.

\medskip

Using this geometric setup we furthermore show that the Vinberg fusion equips, via the induced cospecialization maps, the cohomology of the Zastava spaces with an algebra structure, while the Beilinson-Drinfeld fusion equips it with a coalgebra structure. The associativity of the algebra structure is proven geometrically via a ``double Vinberg degeneration'' over the ``square'' $\BA^1 \times \BA^1$.
Finally, we exploit the fact that both degenerations simultaneously appear in our local models to show that the algebra and coalgebra structures are compatible, i.e., we give a geometric proof that the Hopf algebra axiom is satisfied.

\medskip

It is natural to conjecture that the resulting Hopf algebra agrees with $U(\check{\mathfrak{n}})$, the universal enveloping algebra of the positive part of the Langlands dual Lie algebra $\check{\mathfrak{g}}$.
As is explained in Subsection \ref{A question regarding the Hopf algebra structure} below, it is not hard to see that this identification holds on the level of vector spaces and for the comultiplication map. These identifications, however, are only meaningful if one can also provide a geometric construction of the Langlands dual Lie bracket, i.e., the multiplication map in $U(\check{\mathfrak{n}})$; the Vinberg fusion construction of the present article forms a natural candidate for this. If this holds true, the Vinberg fusion would provide a novel and highly concrete way to pass to the Langlands dual side; this is to be contrasted with the abstract Tannakian approach of the Geometric Satake Equivalence from \cite{MV}. Indeed, in the present framework, elements of $\check{\mathfrak{g}}$ would then be realized as irreducible components of the above moduli spaces, and their Lie bracket could be computed by deforming cartesian products of these components along the Vinberg degeneration. If true, one may then speculate whether this geometric ``enrichment'' of the Langlands dual Lie algebra, and the description of $U(\check{\mathfrak{n}})$ as a cohomological shadow of this geometric enrichment, can provide further insight into Langlands duality.

\medskip

\ssec{Structure of the article}

This article is organized as follows. In Section \ref{Recollections} we review the definition and main properties of the Vinberg semigroup, of the Drinfeld-Lafforgue-Vinberg degeneration $\VinBun_G$, and of the defect stratification discussed above.
In Section \ref{Statements} we introduce the Picard-Lefschetz oscillators for arbitrary reductive groups and state the precise version of Theorem A of this introduction, Theorem \ref{main theorem for nearby cycles} below. In Section \ref{Proofs I} we recall the construction and main properties of the local models. In Section \ref{Proofs II} we give the proof of Theorem \ref{main theorem for nearby cycles}. Finally, in Section \ref{Vinberg fusion and a geometric Hopf algebra structure} we discuss the aforementioned topics related to the Vinberg fusion.

\bigskip
\bigskip

\ssec{Conventions and notation}
\label{Conventions and notation}

\sssec{Sheaves}
We will use a formalism of mixed sheaves. To be concrete, we will work with $\ell$-adic Weil sheaves: We assume the curve $X$ is defined over a finite field, and work with Weil sheaves over the algebraic closure of the finite field. Given a scheme or stack $Y$, we denote by $D(Y)$ its derived category of constructible $\Qellbar$-sheaves. We once and for all fix a square root $\Qellbar(\tfrac{1}{2})$ of the Tate twist $\Qellbar(1)$.
We normalize all IC-sheaves to be pure of weight $0$. In particular, the IC-sheaf of a smooth variety $Y$ is equal to $\Qellbar[\dim Y](\tfrac{1}{2} \dim Y)$. Given a local system $L$ on a smooth dense open subscheme $U$ of a scheme $Y$, we refer to the intermediate extension of the shifted and twisted local system $L[\dim Y](\tfrac{1}{2} \dim Y)$ to $Y$ as \textit{the IC-extension of} $L$.
Our conventions for the nearby cycles functor are stated in Subsection \ref{Recollections about nearby cycles} below.

\medskip

\sssec{Disjoint loci}
We use the symbol $\circ$ to denote the restriction of a scheme, stack, or sheaf to a ``disjoint locus'', to be understood in the appropriate sense depending on the context. As an example, we denote by
$$X^{(n_1)} \stackrel{\circ}{\times} X^{(n_2)}$$
the open subvariety of the product $X^{(n_1)} \times X^{(n_2)}$ of symmetric powers of the curve $X$ obtained by requiring that the two effective divisors have disjoint supports, and refer to it as the disjoint locus of the product $X^{(n_1)} \times X^{(n_2)}$. Similarly, given complexes $F_1 \in D(X^{(n_1)})$ and $F_2 \in D(X^{(n_1)})$ we denote by
$$F_1 \ \overset{\circ}{\boxtimes} \ F_2$$
the restriction of the exterior product $F_1 \boxtimes F_2$ to the disjoint locus of the above product.

\medskip

\sssec{Factorization structures for perverse sheaves}
\label{Factorization structures for perverse sheaves}
Consider the datum of, for each positive integer $n$, a perverse sheaf $P_n$ on the $n$-th symmetric power $X^{(n)}$ of the curve $X$. Denote by
$$\add: \ X^{(n_1)} \times X^{(n_2)} \ \longto \ X^{(n)}$$
the map defined by adding effective divisors on $X$. Then a \textit{factorization structure} on the collection of perverse sheaves $P_n$ is defined as a collection of compatible isomorphisms
$$(\add^* P_n)\big|^*_{X^{(n_1)} \stackrel{\circ}{\times} X^{(n_2)}} \ \ \cong \ \ P_{n_1} \stackrel{\circ}{\boxtimes} P_{n_2}$$
for any $n, n_1, n_2$ with $n_1 + n_2 = n$. We also simply call the collection of perverse sheaves $P_n$ \textit{factorizable} if there is no ambiguity about which factorization structure is being considered. This terminology extends to other parameter spaces indexed by a monoid, such as the spaces $X^{\check\theta}$ indexed by positive coweights $\check\theta \in \Lambdach_G^{pos}$ defined in Subsection \ref{Spaces of effective divisors} below.

\bigskip

\ssec{Acknowledgements}
I would like to express my sincere gratitude to Dennis Gaitsgory and Vladimir Drinfeld for suggesting to study the compactification $\barBun_G$, as well as for their continued encouragement and support. I would also like to thank Michael Finkelberg and Anand Patel for helpful conversations related to the content of this article.

\bigskip
\bigskip
\bigskip
\bigskip
\bigskip
\bigskip

\section{Recollections -- The degeneration $\VinBun_G$}
\label{Recollections}

\bigskip

\ssec{The Vinberg semigroup}

Given any reductive group $G$ of characteristic $0$, E. B. Vinberg (\cite{V}) has constructed a canonical algebraic semigroup, the \textit{Vinberg semigroup} $\Vin_G$ of $G$, which naturally forms a multi-parameter degeneration of $G$ over an affine space. The case of arbitrary characteristic can be found in \cite{Ri1}, \cite{Ri2}, \cite{Ri3}, \cite{Ri4}, and \cite{BKu}. We now sketch the definition of $\Vin_G$ as well as some properties relevant to the present work, which only utilizes a one-parameter sub-family of the multi-parameter family $\Vin_G$ which we call the \textit{principal degeneration} of $G$. We refer the reader to \cite{Sch2} for a discussion of $\Vin_G$ that is not focused on the principal degeneration, and to \cite{Pu}, \cite{Re}, \cite{DrG2} and the above sources for more background and proofs.

\sssec{Notation related to the group}
Let $G$ be a reductive group over $k$. Let $Z_G$ denote the center of $G$ and let $r$ denote the semisimple rank of $G$. For simplicity we assume that the derived group $[G,G]$ of $G$ is simply connected. Fix a maximal torus $T$ of $G$ and a Borel subgroup $B$ containing $T$. Let $N$ denote the unipotent radical of $B$. Let $W$ denote the Weyl group of $G$ and let $w_0$ denote its longest element. We denote by $\Lambda_G$ the weight lattice of $G$, by $\Lambdach_G$ the coweight lattice of $G$, by $\CI$ the set of vertices of the Dynkin diagram of $G$, by $(\alpha_i)_{i \in \CI} \in \Lambda_G$ the simple roots, and by $(\alphacheck_i)_{i \in \CI} \in \Lambdach_G$ the simple coroots. We denote by $\Lambda_G^+$ the set of dominant weights and by $\Lambda_G^{pos}$ the set of positive weights, and analogously for $\Lambdach_G$. We denote by $\leq$ the usual partial order on $\Lambda_G$ and $\Lambdach_G$. Finally, we define the \textit{enhanced group} of $G$ as
$$G_{enh} \ = \ (G \times T) /Z_{G} \, ;$$
here the center $Z_G$ of $G$ acts anti-diagonally on $G \times T$, i.e., via the formula $(g,t).z = (zg, z^{-1}t)$. The inclusion of the first coordinate
$$G \ \longinto \ G_{enh}$$
realizes $G$ as a subgroup of $G_{enh}$.

\medskip

\sssec{Definition of $\Vin_G$ via classification of reductive monoids}
\label{Definition of Vin_G via classification of reductive monoids}
The Vinberg semigroup $\Vin_G$ is an affine algebraic monoid; its group of units is open and dense, and is equal to the reductive group $G_{enh}$. We now recall its definition via the classification of \textit{reductive monoids}, i.e., the classification of irreducible affine algebraic monoids whose group of units is dense, open, and a reductive group.
To do so, denote by $\Rep(G_{enh})$ the category of finite-dimensional representations of the enhanced group $G_{enh}$.
According to the classification of reductive monoids (see \cite{Pu}, \cite{Re}, \cite{V}, \cite{DrG2}), the monoid $\Vin_G$ is uniquely determined by the full subcategory
$$\Rep(\Vin_G) \ \subset \ \Rep(G_{enh})$$
consisting of all representations $V \in \Rep(G_{enh})$ with the property that the $G_{enh}$-action extends to an action of the monoid $\Vin_G$.
To define $\Vin_G$ it thus suffices to specify the full subcategory $\Rep(\Vin_G)$ of $\Rep(G_{enh})$.
To do so, we first introduce the following notation. Any representation $V$ of $G_{enh}$ admits a canonical decomposition as $G_{enh}$-representations
$$V \ = \ \bigoplus_{\lambda \in \Lambda_T} V_\lambda$$
according to the action of the center $Z_{G_{enh}} = (Z_G \times T)/Z_G = T$, i.e., such that the center $Z_{G_{enh}} = T$ acts on the summand $V_\lambda$ via the character $\lambda$. Each summand $V_\lambda$ in this decomposition also naturally forms a $G$-representation via the inclusion $G \into G_{enh}$, whose central character as a $G$-representation equals the restriction $\lambda|_{Z_G}$.
We now define the subcategory $\Rep(\Vin_G)$ of $\Rep(G_{enh})$: A representation $V \in \Rep(G_{enh})$ lies in $\Rep(\Vin_G)$ if and only if for each $\lambda \in \Lambda_T$ the weights of the summand $V_{\lambda}$, considered as a $G$-representation, are all $\leq \lambda$.

\medskip

\sssec{Some first properties}
The Vinberg semigroup $\Vin_G$ is a normal algebraic variety. It comes equipped with a natural $G \times G$-action extending the natural $G \times G$-action on $G_{enh}$. It furthermore comes equipped with a natural $T$-action extending the $T$-action on $G_{enh} = (G \times T) / Z_G$ defined by acting on the second factor. This $T$-action commutes with the $G \times G$-action, and will simply be referred to as \textit{the} $T$-action on $\Vin_G$.

\medskip

We now recall that the Vinberg semigroup $\Vin_G$ forms a canonical multi-parameter degeneration of the group $G$. First, let $T_{adj} = T/Z_G$ denote the adjoint torus of $G$, and recall that the collection of simple roots $(\alpha_i)_{i \in \CI}$ of $G$ give rise to a canonical isomorphism
$$T_{adj} \ \stackrel{\cong}{\longto} \ \BG_m^r \, .$$
Thus the simple roots form canonical affine coordinates on $T_{adj}$. We then obtain a canonical semigroup completion $T_{adj}^+$ of $T_{adj}$ by defining
$$T_{adj}^+ \ := \ \BA^r \ \supset \ \BG_m^r \ = T_{adj} \, ;$$
here the semigroup structure on $\BA^r$ is defined by component-wise multiplication. The natural $T$-action on $T_{adj}$ extends to a $T$-action on $T_{adj}^+$.

\medskip

With this notation, the Vinberg semigroup $\Vin_G$ then admits a natural flat homomorphism of semigroups
$$v: \ \Vin_G \ \longto \ T_{adj}^+ = \BA^r$$
which extends the natural projection map $G_{enh} \longto T_{adj}$ and which is $G \times G$-invariant and $T$-equivariant for the above $T$-actions on $\Vin_G$ and on $T_{adj}^+$. The fiber of this map $v$ over the point $1 \in T_{adj}^+$ is canonically identified with the group $G$; see Subsection \ref{G-locus and B-locus} below for a stronger and more precise statement.

\medskip

\sssec{The canonical section}
\label{The canonical section}
Our fixed choice of a maximal torus $T \subset B \subset G$ gives rise to a section
$$\Fs: \ T_{adj}^+ \ \longto \ \Vin_G$$
of the map
$$v: \ \Vin_G \ \longto \ T_{adj}^+ \, ,$$
which can be uniquely characterized as follows. First note that the map
$$T \ \longto \ G \times T \, , \ \ t \ \longmapsto \ (t^{-1}, t)$$
descends to a map $T_{adj} \longto G_{enh}$; the latter map forms a section of the map $G_{enh} \longto T_{adj}$. Then one can show that this section extends to the desired section $\Fs$ of the map $v$; the image under $\Fs$ of any point in $T_{adj}^+$ in fact lies in the open $G \times G$-orbit of the corresponding fiber of the map $v$.
This implies that the section $\Fs$ factors through the \textit{non-degenerate locus} $\sideset{_0}{_G}\Vin$ of $\Vin_G$, which we recall in the next subsection.

\sssec{The non-degenerate locus}
The Vinberg semigroup contains a natural dense open subvariety $\sideset{_0}{_G}\Vin \subset \Vin_G$, the \textit{non-degenerate locus} of $\Vin_G$; it is characterized uniquely by the fact that it meets each fiber of the map $v: \Vin_G \to T_{adj}^+$ in the open $G \times G$-orbit of that fiber.
The open subvariety $\sideset{_0}{_G}\Vin$ of $\Vin_G$ is in fact not only $G \times G$-stable but also $T$-stable. The restriction of the map $v$ to $\sideset{_0}{_G}\Vin$ is smooth.

\sssec{The stratification parametrized by parabolics}
\label{The stratification parametrized by parabolics}
The completed adjoint torus $T_{adj}^+ = \BA^r$ carries the usual coordinate stratification. Its strata are stable under the action of $T$, and are naturally indexed by subsets of the Dynkin diagram $\CI$ of $G$, or equivalently by standard parabolic subgroups of $G$:
$$T_{adj}^+ \ \ = \ \ \bigcup_{P} \ T^+_{adj, P} \ .$$
Each stratum $T^+_{adj, P}$ of this stratification contains a canonical point $c_P$, as we now recall. Let $\CI_P \subset \CI$ denote the subset of $\CI$ consisting of those vertices corresponding to the parabolic $P$. Then using the canonical identification $T_{adj}^+ = \BA^r$ we define $(c_P)_i = 1$ if $i \in \CI_P$ and $(c_P)_i = 0$ if $i \notin \CI_P$. In particular we have $c_G = 1 \in T_{adj}$ and $c_B = 0 \in T_{adj}^+$.
Via pullback along the map $v$ this stratification of $T_{adj}^+$ induces a stratification
$$\Vin_G \ \ = \ \ \bigcup_{P} \ \Vin_{G,P} \ .$$

\medskip

\sssec{The $G$-locus and the $B$-locus}
\label{G-locus and B-locus}
Note that the $G$-locus $\Vin_{G,G}$ of $\Vin_G$ satisfies
$$\Vin_{G,G} \ = \ G_{enh} \ = \ (G \times T) / Z_G \ = \ G \times T_{adj}$$
as varieties over $T_{adj}$, where the last identification is induced by the map
$$(g,t) \ \mapsto \ (g t^{-1}, t) \, .$$

\medskip

Next we recall a description of the $B$-locus $\Vin_{G,B}$; similar descriptions can be given for the $P$-loci $\Vin_{G,P}$ for arbitrary parabolics $P$ of $G$, but only the case $P=B$ will be needed for the present article.
To describe the $B$-locus $\Vin_{G,B}$, recall first that a scheme Z over $k$ is called {\it strongly quasi-affine} if its ring of global functions $\Gamma(Z, \CO_Z)$ is a finitely generated $k$-algebra and if the natural map
$$Z \ \longto \ \overline{Z} \ := \ \Spec (\Gamma(Z, \CO_Z))$$
is an open immersion. For a strongly quasi-affine variety $Z$ we refer to $\overline{Z}$ as its  {\it affine closure}. With this notation we have:

\medskip

\begin{lemma}
Let the maximal torus $T = B/N$ of $G$ act diagonally on the right on the product $G/N \times G/N^-$. Then the quotient
$$(G/N \times G/N^-)/T$$
is strongly quasi-affine.
\end{lemma}

\medskip

We denote by $\overline{(G/N \times G/N^-)/T}$ the corresponding affine closure. We can then recall (see e.g. \cite[Sec. 4.2]{W1}):
\medskip

\begin{lemma}
\label{Vinberg lemma}
There exists a canonical isomorphism
$$\overline{(G/N \times G/N^-)/T} \ \ \ \stackrel{\cong}{\longto} \ \ \ \Vin_G|_{c_B} = \Vin_{G,B}$$
which is $G \times G$-equivariant for the natural $G \times G$-actions and which restricts to an isomorphism
$$(G/N \times G/N)/T \ \ \ \stackrel{\cong}{\longto} \ \ \ \sideset{_0}{_G}\Vin|_{c_B} \, .$$
\end{lemma}

\bigskip

\sssec{The example $G = \SL_2$}
\label{The case G = SL_2}

For $G = \SL_2$ the Vinberg semigroup $\Vin_G$ is equal to the semigroup of $2 \times 2$ matrices~$\Mat_{2 \times 2}$. The $\SL_2 \times \SL_2$-action is given by left and right multiplication; the action of $T = \BG_m$ is given by scalar multiplication. The homomorphism of semigroups $v$ is equal to the determinant map
$$v: \ \Vin_G = \Mat_{2 \times 2} \ \stackrel{\det}{\longto} \ \BA^1 = T_{adj}^+ \, .$$
In particular we find that
$$\Vin_{G,G} \ = \ v^{-1}(\BA^1 \setminus \{ 0\}) \ \cong \ \GL_2 \, ,$$
and that the $B$-locus
$\Vin_{G,B} = v^{-1}(0)$ consists of all singular $2 \times 2$ matrices. The non-degenerate locus $\sideset{_0}{_G}\Vin$
is equal to the subset
of non-zero matrices
$$\Mat_{2 \times 2} \setminus \{0\} \ \ \subset \ \ \Mat_{2 \times 2}.$$

\bigskip

\ssec{The degeneration $\VinBun_G$}

We can now recall the definition of the Drinfeld-Lafforgue-Vinberg degeneration $\VinBun_G$ from \cite{Sch2}; the definition of the Drinfeld-Lafforgue-Vinberg compactification $\barBun_G$, which is also given in \cite{Sch2} and of which $\VinBun_G$ is a minor modification, is due to Drinfeld (unpublished).

\sssec{Notation}
Let $G$ be a reductive group over $k$ and let $X$ be a smooth projective curve over $k$.
For any stack $\CY$ the sheaf of groupoids $\Maps(X,\CY)$ parametrizing maps from $X$ to $\CY$ is defined as
$$\Maps(X, \CY)(S) \ = \ \CY(X \times S) \, .$$
For example, we have $\Bun_G = \Maps(X, \cdot/G)$.
Next, for an open substack $\overset{\circ}{\CY} \subset \CY$, the sheaf of groupoids
$\Maps_{gen}(X, \CY \supset \overset{\circ}{\CY})$
assigns to a scheme $S$ the full sub-groupoid of $\Maps (X, \CY)(S)$ consisting of all maps $X \times S \to \CY$ satisfying the following condition: We require that for every geometric point $\bar s \to S$ there exists an open dense subset of $X \times \bar{s}$ on which the restricted map $X \times \bar s \to \CY$ factors through the open substack~$\overset{\circ}{\CY} \subset \CY$. 

\medskip

\sssec{Definition of $\VinBun_G$}
Consider the open substack
$$\sideset{_0}{_G}\Vin / G \times G \ \ \ \subset \ \ \ \Vin_G / G \times G$$
obtained by quotienting out by the $G \times G$-action.
Then the \textit{Drinfeld-Lafforgue-Vinberg degeneration} $\VinBun_G$ is defined as
$$\VinBun_G \ \ := \ \ \Maps_{gen} \, (X, \ \text{Vin}_G / G \times G \ \supset \ \sideset{_0}{_G}\Vin / G \times G) \, .$$
As the curve $X$ is assumed to be proper, the map $v: \Vin_G \longto T_{adj}^+$ induces a map
$$v: \ \VinBun_G \ \longto \ T_{adj}^+  = \BA^r$$
which makes $\VinBun_G$ into a multi-parameter degeneration of $\Bun_G$: Any fiber of the map $v$ over a point in $T_{adj} \subset T_{adj}^+$ is isomorphic to $\Bun_G$.

\medskip

\sssec{The example $G = \SL_2$}
For $G=\SL_2$ an $S$-point of $\VinBun_G$ consists of the data of two vector bundles $E_1$, $E_2$ of rank $2$ on $X \times S$, together with trivializations of their determinant line bundles $\det E_1$ and $\det E_2$, and a map of coherent sheaves
$$\varphi: \ E_1 \ \longto \ E_2 \, ,$$
satisfying the condition that for each geometric point $\bar{s} \to S$ the map
$$\varphi|_{X \times \bar{s}} : \ \ E_1|_{X \times \bar{s}} \ \longto \ E_2|_{X \times \bar{s}}$$
is not the zero map. In other words, for each geometric point $\bar{s} \to S$ the map $\varphi|_{X \times \bar{s}}$ is required to not vanish generically on the curve $X \times \bar{s}$. The map $v: \VinBun_G \longto \BA^1$ is obtained by sending the above data to the point $\det(\varphi) \in \BA^1(S)$.

\medskip

\sssec{The defect-free locus of $\VinBun_G$}
\label{The defect-free locus of VinBun_G}
The \textit{defect-free locus} of $\VinBun_G$ is defined as the open substack
$$\sideset{_0}{_G}\VinBun \ \ := \ \ \Maps(X, \sideset{_0}{_G}\Vin / G \times G) \, .$$
From Lemma \ref{Vinberg lemma} above we see that its fiber over the point $c_B = 0$ equals
$$\sideset{_0}{_{G}}\VinBun |_{c_B} \ \ = \ \ \Bun_{B^-} \underset{\Bun_T}{\times} \Bun_B \, .$$
It is not hard to show that the restriction of the map $v$ to the defect-free locus
$$v: \ \sideset{_0}{_G}\VinBun \ \ \longto \ \ T_{adj}^+$$
is smooth. In particular, the defect-free locus $\sideset{_0}{_G}\VinBun$ itself is smooth.

\medskip

\sssec{Stratification by parabolics}
\label{Stratification by parabolics}
Since $T_{adj}^+$ carries a natural stratification indexed by parabolic subgroups $P$ of $G$, we obtain via pullback along the map $v$ an analogous stratification
$$\VinBun_G \ \ = \ \ \bigcup_{P} \ \VinBun_{G,P} \, .$$
Only the strata $\VinBun_{G,G}$ and $\VinBun_{G,B}$ will appear in the present work. It is not hard to see that the $G$-locus $\VinBun_{G,G}$ forms a canonically trivial fiber bundle over $T_{adj}$:
$$\VinBun_{G,G} \ \ = \ \ \Bun_G \times \, T_{adj}$$
We will introduce a \textit{defect stratification} of the $B$-locus $\VinBun_{G,B}$ in Subsection \ref{The defect stratification} below. For analogous stratifications of the strata $\VinBun_{G,P}$ for arbitrary proper parabolics $P$ we refer the reader to \cite[Sec. 3]{Sch2}.

\bigskip

\ssec{The defect stratification of the $B$-locus}
\label{The defect stratification}

We now recall the aforementioned \textit{defect stratification} of the $B$-locus $\VinBun_{G,B} = \VinBun_G|_{c_B}$.

\sssec{The monoid $\overline{T}$}
\label{The monoid barT}

First we review the definition of a certain monoid $\overline{T}$ containing the maximal torus $T$ as a dense open subgroup; we refer the reader to \cite{BG1}, \cite{W1}, and \cite{Sch2} for proofs, additional background, and the case of an arbitrary Levi subgroup.
First recall from e.g. \cite{BG1} that the quotient $G/N$ is strongly quasi-affine, and denote its affine closure by $\overline{G/N}$. The monoid $\overline{T}$ is then defined as the closure of $T$ inside $\overline{G/N}$ under the embedding
$$T = B/N \ \longinto \ G/N \ \subset \ \overline{G/N} \, .$$
The $T$-actions from the left and right on $G/N$ give rise to $T$-actions from the left and right on $\overline{T}$; these $T$-actions in turn extend to $\overline{T}$-actions, so that $\overline{T}$ indeed forms an algebraic monoid containing the group $T$. One can also define $\overline{T}$ as follows: Instead of the tautological embedding of $T = B^-/N^-$ into $G/N^-$, consider the embedding given by the inverse:
$$T \ \longinto \ G/N^- \, , \ \ \ t \ \longmapsto \ t^{-1}$$
One can then also define $\overline{T}$ as the closure of $T$ under this embedding into $\overline{G/N^-}$.

\medskip

\sssec{The embedding of $\overline{T}$ into $\Vin_G$}
\label{The embedding of barT into Vin_G}
Recall that the embedding of the first factor
$$G/N \ \longinto \ (G/N \times G/N^-)/T$$
and the embedding of the second factor
$$G/N^- \ \longinto \ (G/N \times G/N^-)/T$$
extend to closed immersions
$$\overline{G/N} \ \ \longinto \ \ \overline{(G/N \times G/N^-)/T}$$
and
$$\overline{G/N^-} \ \ \longinto \ \ \overline{(G/N \times G/N^-)/T} \, .$$
Then one can show that the two closed embeddings
$$\overline{T} \ \ \longinto \ \ \overline{(G/N \times G/N^-)/T} \ = \ \Vin_{G,B}$$
of $\overline{T}$ obtained by composing the previous embeddings with the embeddings of $\overline{T}$ into $\overline{G/N}$ and $\overline{G/N^-}$ from Subsection \ref{The monoid barT} above coincide. This embedding is $T \times T$-equivariant for the natural $T \times T$-action on $\overline{T}$ and the $T \times T$-action on $\overline{(G/N \times G/N^-)/T}  = \Vin_{G,B}$ obtained by restricting the $G \times G$-action to the subgroup $T \times T$.

\medskip

\sssec{Spaces of effective divisors}
\label{Spaces of effective divisors}
For any positive integer $n$ we denote the $n$-th symmetric power of the curve $X$ by $X^{(n)}$. Given a positive coweight $\thetacheck  = \sum_{i \in \CI} n_i \alphacheck_i \in \Lambdach_G^{pos}$ of $G$ we define
$$X^{\thetacheck} \ \ = \ \ \prod_{i \in \CI} X^{(n_i)} \, .$$
As a variety, the space $X^{\thetacheck}$ is a partially symmetrized power of the curve $X$. It can be thought of as the space of $\Lambdach_G^{pos}$-valued divisors on $X$, i.e., as the space of formal linear combinations $\sum_{k} \thetacheck_k x_k$ with $x_k \in X$ and $\thetacheck_k \in \Lambdach_G^{pos}$ satisfying $\sum_k \thetacheck_k = \thetacheck$.
The spaces $X^{\check\theta}$ constitute the connected components of the mapping stack $\Maps_{gen}(X, \overline{T}/T \supset T/T = pt)$:
$$\Maps_{gen}(X, \overline{T}/T \supset T/T = pt) \ \ = \ \ \bigcup_{\check\theta \in \Lambdach_G^{pos}} X^{\check\theta} \, .$$

\medskip

\sssec{Strata maps}
\label{Strata maps}
The closed immersion
$$\overline{T} \ \ \longinto \ \ \overline{(G/N \times G/N^-)/T} \ = \ \Vin_{G,B}$$
from Subsection \ref{The embedding of barT into Vin_G} above induces a map of quotient stacks
$$\overline{T} / (B \times B^-) \ \ \longto \ \ \Bigl( \overline{(G/N \times G/N^-)/T} \Bigr) / (G \times G) \, .$$
which by Lemma \ref{Vinberg lemma} in turn induces the desired strata map
$$f: \ \Maps_{gen}(X, \, \overline{T} / (B \times B^-) \, \supset \, T / (B \times B^-)) \ \ \longto \ \ \VinBun_{G,B} \, .$$
To describe the source of the map $f$ more explicitly, first note that the quotient stack $ \overline{T} / (B \times B^-)$ can be rewritten as
$$\overline{T} / (B \times B^-) \ \ = \ \ \cdot / B^{-} \ \underset{\cdot / T}{\times} \ \overline{T} / (T \times T) \ \underset{\cdot / T}{\times} \ \cdot / B \ \ = \ \ \cdot / B^{-} \ \underset{\cdot / T}{\times} \ \bigl( \overline{T} / T \ \times \ \cdot / B \bigr) \, ,$$
where the map $\overline{T} / T \times \cdot / B \to \cdot / T$ used in the last fiber product factors as
$$\overline{T} / T \ \times \ \cdot / B \ \stackrel{\text{forget}}{\longto} \ \cdot/T \times \cdot/T \ \stackrel{\text{multiply}}{\longto} \ \cdot / T \, .$$
Thus the source of the map $f$ decomposes into a disjoint union of connected components
$$\bigcup_{(\lambdach_1, \lambdach_2, \check\theta)} \Bun_{B^-, \lambdach_1} \ \underset{\Bun_T}{\times} \ \bigl( X^{\check\theta} \ \times \ \Bun_{B,\lambdach_2} \bigr) \, ,$$
where $\lambdach_1, \lambdach_2 \in \Lambdach_G = \pi_0(\Bun_{B^-}) = \pi_0(\Bun_{B})$ and $\lambdach_2 - \check\theta = \lambdach_1$. Here the map $X^{\check\theta} \times \Bun_{B,\lambdach_2} \longto \Bun_T$ used to define the fiber product factors as
$$X^{\check\theta} \ \times \ \Bun_{B,\lambdach_2} \ \stackrel{\text{forget}}{\longto} \ X^{\check\theta} \ \times \ \Bun_{T,\lambdach_2} \ \stackrel{\text{twist}}{\longto} \ \Bun_{T, \lambdach_2 - \check\theta}$$
where by \textit{twist} we denote the usual operation of twisting a $T$-bundle by a $\Lambdach_G^{pos}$-valued divisor.
We will denote by $f_{\lambdach_1, \check\theta, \lambdach_2}$ the restriction of $f$ to the connected component corresponding to the triple $(\lambdach_1, \check\theta, \lambdach_2)$ in the above decomposition.

\medskip

It is shown in \cite[Proposition 3.3.2]{Sch2}:

\medskip

\begin{proposition}
\label{defect stratification proposition}
The maps $f_{\lambdach_1, \check\theta, \lambdach_2}$ are locally closed immersions. We will denote the corresponding locally closed substacks by
$$\sideset{_{\lambdach_1, \check\theta, \lambdach_2}}{_{G,B}}\VinBun \ \longinto \ \VinBun_{G,B} \, .$$
Furthermore, the locally closed substacks $\sideset{_{\lambdach_1, \check\theta, \lambdach_2}}{_{G,B}}\VinBun$ form a stratification of $\VinBun_{G,B}$ in the following sense: On the level of $k$-points, the stack $\VinBun_{G,B}$ is equal to the disjoint union
$$\VinBun_{G,B} \ \ = \ \ \bigcup_{(\lambdach_1, \check\theta, \lambdach_2)} \ \sideset{_{\lambdach_1, \check\theta, \lambdach_2}}{_{G,B}}\VinBun \, ,$$
where the union runs over all $\lambdach_1, \lambdach_2 \in \Lambdach_G$ and $\check\theta \in \Lambdach_G^{pos}$ such that $\lambdach_2 - \check\theta = \lambdach_1$.
\end{proposition}

\medskip

\sssec{Defect value and defect}
Each stratum
$$\sideset{_{\lambdach_1, \check\theta, \lambdach_2}}{_{G,B}}\VinBun \ \ = \ \ \Bun_{B^-, \lambdach_1} \ \underset{\Bun_T}{\times} \ \bigl( X^{\check\theta} \ \times \ \Bun_{B,\lambdach_2} \bigr)$$
of $\VinBun_{G,B}$ comes equipped with a forgetful map to the space $X^{\check\theta}$. Given a $k$-point of $\VinBun_{G,B}$ lying in this stratum, the corresponding $k$-point of $X^{\check\theta}$ will be referred to as its \textit{defect value}, and the positive coweight $\check\theta \in \Lambdach_G^{pos}$ as its \textit{defect}.

\bigskip

\ssec{Compactifying the strata maps}
\label{Compactifying the strata maps}

In this subsection we recall natural compactifications $\bar f_{\lambdach_1, \check\theta, \lambdach_2}$ of the strata maps $f_{\lambdach_1, \check\theta, \lambdach_2}$ constructed in \cite{Sch2}. To do so, we first briefly recall Drinfeld's relative compactification $\barBun_B$; we refer the reader to \cite{BG1}, \cite{BFGM} and \cite{Sch3} for proofs and further background on $\barBun_B$.

\medskip

\sssec{Drinfeld's relative compactification $\barBun_B$}
\label{Drinfeld's relative compactification}

The space $\barBun_B$ can be defined as the mapping stack
$$\barBun_B \ := \ \Maps_{gen}(X, \, G \backslash \overline{G/N} / T \, \supset \, \cdot / B) \, .$$
It naturally contains $\Bun_B$ as a dense open substack, and the schematic map $\Bun_B \to \Bun_G$ extends to a schematic map
$$\barBun_B \ \longto \ \Bun_G$$
which is proper when restricted to any connected component $\barBun_{B, \check\lambda}$ of $\barBun_B$, where $\check\lambda \in \pi_0(\barBun_B) = \Lambdach_G$.

\medskip

The space $\barBun_B$ admits the following stratification. The torus action
$$G/N \times T \ \longto \ G/N$$
extends to an action of the monoid $\overline{T}$
$$\overline{G/N} \times \overline{T} \ \longto \ \overline{G/N} \, ,$$
and the latter in turn induces natural maps
$$X^{\check\theta} \times \barBun_{B, \check\lambda + \check\theta} \ \ \longto \ \ \barBun_{B, \check\lambda}$$
for any $\check\lambda \in \Lambdach_G$ and $\check\theta \in \Lambdach_G^{pos}$.
It is then shown in \cite{BG1} that the restricted maps
$$X^{\check\theta} \times \Bun_{B, \check\lambda + \check\theta} \ \ \longto \ \ \barBun_{B, \check\lambda}$$
are locally closed immersions, and that they stratify $\barBun_{B, \check\lambda}$ as $\check\theta$ ranges over the set $\Lambdach_G^{pos}$:
$$\barBun_{B, \check\lambda} \ \ = \ \ \bigcup_{\check\theta \in \Lambdach_G^{pos}} \ X^{\check\theta} \times \Bun_{B, \check\lambda + \check\theta}$$

\medskip

\sssec{Compactifying the maps $\bar f_{\lambdach_1, \check\theta, \lambdach_2}$}

We now recall the construction of the compactified maps $\bar f_{\lambdach_1, \check\theta, \lambdach_2}$ from \cite{Sch2}.
By Subsection \ref{The embedding of barT into Vin_G} above, the $B$-locus of the Vinberg semigroup
$$\Vin_G|_{c_B} \ \ = \ \ \Vin_{G,B} \ \ = \ \ \overline{(G/N \times G/N^-)/T}$$
naturally contains the varieties $\overline{G/N}$, $\overline{T}$, and $\overline{G/N^-}$ as subvarieties. As the inverse image of $0 \in T_{adj}^+$ under a semigroup homomorphism, the $B$-locus furthermore carries a structure of semigroup (without unit), referred to as the \textit{asymptotic semigroup} in the literature. Using the multiplication operation of this semigroup we obtain a map
$$\overline{G/N} \times \overline{T} \times \overline{G/N^-} \ \ \longto \ \ \Vin_{G,B} = \overline{(G/N \times G/N^-)/T}$$
by multiplying the three subvarieties. One can alternatively also obtain this map by first acting by $\overline{T}$ on either $\overline{G/N}$ or $\overline{G/N^-}$, and then multiply in $\Vin_{G,B}$ with the remaining subvariety.

\medskip

The above map then induces the desired maps
$$\bar f_{\lambdach_1, \check\theta, \lambdach_2}: \ \barBun_{B^-, \lambdach_1} \ \underset{\Bun_T}{\times} \ \bigl( X^{\check\theta} \ \times \ \barBun_{B,\lambdach_2} \bigr) \ \ \longto \ \VinBun_{G,B}$$
which extend the strata maps $f_{\lambdach_1, \check\theta, \lambdach_2}$ from Subsection \ref{Strata maps} above and which are finite: Indeed, it follows from the properness of $\barBun_B$ and $\barBun_{B^-}$ that they are proper, and the quasi-finiteness of the addition map of effective divisors implies that they are also quasi-finite.

\bigskip
\bigskip
\bigskip
\bigskip
\bigskip

\section{Statements -- Main theorem about nearby cycles}
\label{Statements}

\bigskip

\ssec{Recollections about nearby cycles}
\label{Recollections about nearby cycles}

\sssec{Notation}

For any scheme or stack $Y$ equipped with a map $Y \to \BA^1$ we denote by
$$\Psi: \ \D(Y|_{\BA^1 \setminus \{0\}}) \ \longto \ \D(Y|_{\{0\}})$$
the unipotent nearby cycles functor in the perverse and Verdier-self dual renormalization; it differing from the usual unipotent nearby cycles functor by the shift and twist $[-1](-\tfrac{1}{2})$. With this convention the functor $\Psi$ is t-exact for the perverse t-structure and commutes with Verdier duality literally and not just up to twist. We simply refer to $\Psi$ as \textit{the nearby cycles}.
We denote the logarithm of the unipotent part of the monodromy operator by
$$N: \ \Psi \ \longto \ \Psi(-1) \, ,$$
and simply refer to it as \textit{the monodromy operator}.
We refer the reader to \cite{B} and \cite[Sec. 5]{BB} for additional background on unipotent nearby cycles.

\bigskip

\sssec{Monodromy and weight filtrations}
We now recall some definitions and facts about the monodromy and weight filtrations on nearby cycles, referring the reader to \cite[Sec. 1.6]{De} and \cite[Sec. 5]{BB} for proofs.

\medskip

Given any perverse sheaf $F$ on $Y|_{\BA^1 \setminus \{0\}}$, the operator $N$ by construction acts nilpotently on the perverse sheaf $\Psi(F)$. It therefore induces the \textit{monodromy filtration} on $\Psi(F)$, i.e., the unique finite filtration
$$\Psi(F) = M_{n} \ \supseteq \ M_{n-1} \ \supseteq \ \cdots \ \supseteq \ M_{-n} \ \supseteq 0$$
by perverse sheaves $M_i$ satisfying that
$$N(M_i) \ \subset \ M_{i-2}(-1)$$
for all $i$, and that the induced maps
$$N^i: \ M_i/M_{i-1} \ \longto \ \bigl(M_{-i}/M_{-i-1}\bigr)(-i)$$
are isomorphisms for all $i \geq 0$.
The operator $N$ thus also acts on the associated graded perverse sheaf $\gr(\Psi(F))$, and we have the following well-known linear-algebraic lemma (the Jacobson-Morozov theorem):

\medskip

\begin{lemma}
\label{Lefschetz-sl_2}
The action of the monodromy operator $N$ on the associated graded $\gr(\Psi(F))$ canonically extends to an action of the ``Lefschetz-$\sl_2$'' on $\gr(\Psi(F))$, i.e.: There exists a unique action of the Lie algebra $\sl_2(\overline{\BQ}_\ell)$ on $\gr(\Psi(F))$ such that the action of the lowering operator of $\sl_2(\overline{\BQ}_\ell)$ agrees with the action of $N$, and such that the Cartan subalgebra of $\sl_2(\overline{\BQ}_\ell)$ acts on the summand $\gr(\Psi(F))_i = M_i/M_{i-1}$ with Cartan weight $i$. Thus the decomposition
$$\gr(\Psi(F)) \ = \ \bigoplus_{i} M_i/M_{i-1}$$
agrees with the decomposition of the $\sl_2(\overline{\BQ}_\ell)$-representation $\gr(\Psi(F))$ according to Cartan weights. We will refer to the Lie algebra $\sl_2(\overline{\BQ}_\ell)$ in this context as \textit{the Lefschetz}-$\sl_2$.
\end{lemma}

\medskip

If the perverse sheaf $F$ is pure, Gabber has shown:

\begin{proposition}[Gabber]
\label{Gabber's theorem}
Let $F$ be a pure perverse sheaf of weight~$0$. Then the subquotients of the monodromy filtration on $\Psi(F)$ are also pure, and the weight of the subquotient $\gr(\Psi(F))_i  = M_i/M_{i-1}$ is equal to~$i$. I.e., the monodromy filtration and the weight filtration $\Psi(F)$ agree, and the weight of each subquotient as a Weil sheaf agrees with its Cartan weight with respect to the action of the Lefschetz-$\sl_2$.
\end{proposition}

\bigskip

\ssec{Picard-Lefschetz oscillators for arbitrary reductive groups}
\label{PLO}

We now recall the definition of the \textit{Picard-Lefschetz oscillators} from \cite{Sch1} in the case $G=\SL_2$, and then give a definition of Picard-Lefschetz oscillators for an arbitrary reductive group $G$; these sheaves will in fact depend on the Langlands dual group $\check G$ of $G$. We first review:

\medskip

\sssec{External exterior powers}
\label{sssec External exterior powers}

Let $E$ be a local system on the curve $X$, placed in cohomological degree $0$. The $n$-\textit{th} \textit{external exterior power} $\Lambda^{(n)}(E)$ of $E$ on the symmetric power of the curve $X^{(n)}$ is defined as follows: Note first that the $n$-fold external product $E \boxtimes \cdots \boxtimes E$ on the $n$-th power $X^n$ comes equipped with a natural equivariant structure for the action of the symmetric group $S^n$ on $X^n$; its pushforward $p_* (E \boxtimes \cdots \boxtimes E)$ along the natural map
$$p: \ X^n \ \longto \ X^{(n)}$$
is therefore equipped with an $S^n$-action. One then obtains the $n$-th external exterior power $\Lambda^{(n)}(E)$ by taking the $S^n$-invariants of the pushforward $p_* (E \boxtimes \cdots \boxtimes E)$ against the sign character of $S^n$.
The external exterior power construction is functorial and satisfies (see for example \cite[Sec. 5]{G1}):

\medskip

\begin{lemma}
\label{external exterior powers}
\begin{itemize}
\item[]
\item[(a)] The restriction of the $n$-th external exterior power $\Lambda^{(n)}(E)$ to the disjoint locus $\overset{\circ}{X} {}^{(n)}$ is again a local system.
\item[(b)] The shifted object $\Lambda^{(n)}(E)[n]$ is a perverse sheaf; it is equal to the intermediate extension of its restriction to the disjoint locus.
\item[(c)] The collection of perverse sheaves $\Lambda^{(n)}(E)[n]$ is factorizable, in the sense of Subsection \ref{Factorization structures for perverse sheaves} above.
\end{itemize}
\end{lemma}

\bigskip

\sssec{Picard-Lefschetz oscillators for $G=\SL_2$}
\label{PLOs for SL_2}

Following \cite{Sch1}, we denote by
$$V \ = \ \Qellbar(\tfrac{1}{2}) \oplus \Qellbar(-\tfrac{1}{2})$$
the $2$-dimensional standard representation of the Lefschetz-$\sl_2$, and let
$$\underline V \ := \ V \otimes \Qellbar_X$$
denote the corresponding constant local system of rank $2$ on the curve $X$ together with the induced action of the Lefschetz-$\sl_2$. The \textit{Picard-Lefschetz oscillator} $\CP_n$ on $X^{(n)}$ is then defined as the $n$-th external exterior power of $\underline V \,$, shifted and twisted in the following way:
$$\CP_n \ := \ \Lambda^{(n)} (\underline V) \, [n](\tfrac{n}{2})$$

\medskip

Lemma \ref{external exterior powers} above shows that $\CP_n$ is a perverse sheaf on $X^{(n)}$, equipped with an action of the Lefschetz-$\sl_2$.
Lemma \ref{external exterior powers} also shows that the Picard-Lefschetz oscillators can be obtained as IC-extensions in the following manner: Let the symmetric group $S^n$ act on the $n$-fold tensor power $V \otimes \cdots \otimes V$ by both permuting the factors and also multiplying by the sign of the permutation, and consider the local system on the disjoint locus $\overset{\circ}{X} {}^{(n)}$ corresponding to this action. Then the IC-extension of this local system is equal to $\CP_n$. In particular the perverse sheaf $\CP_n$ is semisimple. Furthermore, the factorization structure on the collection of Picard-Lefschetz oscillators $\CP_n$ respects the action of the Lefschetz-$\sl_2$.

\bigskip

Below we define the correct generalizations of the Picard-Lefschetz oscillators for arbitrary reductive groups. These are certain perverse sheaves $\CF_{\thetacheck}$ on the spaces $X^{\thetacheck}$ which are built from the sheaves $\CP_n$ above in a combinatorial fashion depending on the Langlands dual group $\check G$ of $G$. We first need to recall:

\medskip

\sssec{Kostant partitions}

For any positive coweight $\thetacheck \in \Lambdach_G^{pos}$ we define a \textit{Kostant partition} of $\thetacheck$ to be a collection of non-negative integers $(n_{\betacheck})_{\betacheck \in \check R^+}$ indexed by the set of positive coroots $\check R^+$ of $G$, satisfying that
$$\thetacheck \ = \ \sum_{\betacheck \in \check R^+} n_{\betacheck} \betacheck \, .$$
Put differently, a Kostant partition of $\thetacheck$ is a partition $\thetacheck = \sum_k \thetacheck_k$ of $\thetacheck$ where each summand $\thetacheck_k$ is required to be a positive coroot of $G$.
We will simply refer to the expression $\thetacheck = \sum_{\betacheck \in \check R^+} n_{\betacheck} \betacheck$ as a Kostant partition of $\thetacheck$.
We will denote the finite set of all Kostant partitions of $\thetacheck$ by $\Kost(\thetacheck)$. Note that the cardinality of the set $\Kost(\thetacheck)$ is by definition the value of the Kostant partition function of the Langlands dual group $\check G$ evaluated at the weight $\thetacheck \in \Lambdach_G = \Lambda_{\check G}$ of $\check G$.

\medskip

\sssec{Picard-Lefschetz oscillators for arbitrary reductive groups}
To any Kostant partition
$$\CK: \ \ \thetacheck \ = \ \sum_{\betacheck \in \check R^+} n_{\betacheck} \betacheck$$
of a positive coweight $\thetacheck \in \Lambdach_G^{pos}$ we associate the partially symmetrized power
$$X^{\CK} \ \ := \ \ \prod_{\betacheck \in \check R^+} X^{(n_{\betacheck})}$$
of the curve $X$.
We denote by
$$i_{\CK}: \ \ X^{\CK} \ \longto \ X^{\thetacheck}$$
the finite map defined by adding $\Lambdach_G^{pos}$-valued divisors.
We furthermore define the perverse sheaf
$$\CP_{\CK} \ \ := \ \ \underset{\betacheck \in \check R^+}{\boxtimes} \, \CP_{n_{\betacheck}}$$
on the partially symmetrized power $X^{\CK}$, where $\CP_{n_{\betacheck}}$ denotes the Picard-Lefschetz oscillator on the symmetric power~$X^{(n_{\betacheck})}$.

\medskip

Finally, for any positive coweight $\check\theta \in \Lambdach_G^{pos}$ we define the \textit{Picard-Lefschetz oscillator} $\CF_{\thetacheck}$ on $X^{\thetacheck}$ as the direct sum
$$\CF_{\thetacheck} \ \ := \ \ \bigoplus_{\CK \, \in \, \Kost(\check\theta)} i_{\CK, *} \, \CP_{\CK} \, .$$
By construction the perverse sheaf $\CF_{\thetacheck}$ comes equipped with an action of the Lefschetz-$\sl_2$.

\bigskip

\ssec{The main theorem}

\sssec{The principal degeneration $\VinBun_G^{princ}$}

We will now focus our attention on the restriction of the multi-parameter family $\VinBun_G \to T_{adj}^+ = \BA^r$ to a \textit{general} line passing through the origin $c_B = 0$, i.e., to a line through the origin which is not contained in any of the coordinate planes of $\BA^r$. Since the $T_{adj}$-action on $T_{adj}^+$ lifts to a $T_{adj}$-action on $\VinBun_G$, one obtains isomorphic families regardless of the choice of the general line. For concreteness, we thus restrict $\VinBun_G$ to the line $L_B$ passing through the origin $c_B = 0$ and the point $c_G = 1 \in T_{adj} \subset T_{adj}^+$. We denote the resulting one-parameter family by
$$\VinBun_G^{princ} \ \longto \ L_B = \BA^1$$
and refer to it as the \textit{principal degeneration} of $\Bun_G$. Here we identify the point $0 \in \BA^1$ with $0 \in T_{adj}^+$ and the point $1 \in \BA^1$ with the point $1 \in T_{adj} \subset T_{adj}^+$. 

\medskip

The principal degeneration $\VinBun_G^{princ}$ consists of the $G$-locus
$$\VinBun_{G,G}^{princ} \ \ \ = \ \ \ \Bun_G \ \times \ \ \bigl( \BA^1 \setminus \{0\} \bigr)$$
and the $B$-locus
$$\VinBun_{G,B}^{princ} \ \ = \ \ \VinBun_{G,B} \ \ = \ \ \VinBun_G|_{c_B} \, .$$ 
While we will focus on the principal degeneration of $\VinBun_G$ for the present article, we remark that restrictions of $\VinBun_G$ to other lines in $T_{adj}^+ = \BA^r$ involving the parabolic strata can be dealt with in a similar fashion.

\bigskip

\sssec{The nearby cycles theorem}

To state our main theorem about the nearby cycles of the principal degeneration $\VinBun_G^{princ}$, we denote by
$\IC_{\VinBun_{G,G}^{princ}}$ the IC-sheaf of the $G$-locus of $\VinBun_G^{princ}$; it is a constant sheaf shifted and twisted according to our conventions in Subsection \ref{Conventions and notation} above. Furthermore, we denote by
$$\IC_{\barBun_{B^-, \lambdach_1}} \underset{ \ \Bun_T}{\boxtimes} \Bigl( \CF_{\check\theta} \ \boxtimes \ \IC_{\barBun_{B, \lambdach_2}} \Bigr)$$
the $*$-restriction of the external product
$$\IC_{\barBun_{B^-, \lambdach_1}} \boxtimes \ \Bigl( \CF_{\check\theta} \, \boxtimes \ \IC_{\barBun_{B, \lambdach_2}} \Bigr)$$
from the product space to the fiber product
$$\barBun_{B^-, \lambdach_1} \ \underset{\Bun_T}{\times} \ \bigl( \, X^{\check\theta} \, \times \barBun_{B^-, \lambdach_2} \bigr) \, ,$$
shifted by $[- \dim \Bun_T]$ and twisted by $(- \tfrac{\dim \Bun_T}{2})$. We then have:

\bigskip

\begin{theorem}
\label{main theorem for nearby cycles}
There exists an isomorphism of perverse sheaves
$$\gr \, \Psi (\IC_{\VinBun_{G,G}^{princ}})     \ \ \cong \ \bigoplus_{(\lambdach_1, \thetacheck, \lambdach_2)} \bar{f}_{\lambdach_1, \thetacheck, \lambdach_2, *} \ \Bigl( \IC_{\barBun_{B^-}, \lambdach_1} \underset{ \, \Bun_T}{\boxtimes} \Bigl( \CF_{\thetacheck} \ \boxtimes \ \IC_{\barBun_{B, \lambdach_2}} \Bigr) \Bigr)$$
which identifies the action of the Lefschetz-$\sl_2$ on the right hand side with the monodromy action on the left hand side. Here the direct sum runs over all triples $(\lambdach_1, \thetacheck, \lambdach_2)$ with $\lambdach_1, \lambdach_2 \in \Lambdach_G$, $\thetacheck \in \Lambdach_G^{pos}$, and $\lambdach_1 + \check\theta = \lambdach_2$.
\end{theorem}

\bigskip
\bigskip
\bigskip
\bigskip
\bigskip

\section{Proofs I --- Local models}
\label{Proofs I}

\bigskip

\ssec{Construction of local models}
\label{Construction of local models}

We now recall the construction of certain \textit{local models} for $\VinBun_G$ from \cite{Sch2}. In \cite{Sch2}, one such local model is constructed for each proper parabolic $P$ of $G$, and then used to study the singularities of $\VinBun_G$ lying in the $P$-locus $\VinBun_{G,P}$. Since the present article is only concerned with the principal degeneration $\VinBun_G^{princ}$, whose singularities all lie in the $B$-locus $\VinBun_{G,B}$, only the local model for the Borel $B$ will be needed. This local model also naturally forms a family over the affine space $T_{adj}^+ = \BA^r$, and we restrict it to the line $L_B = \BA^1$ in $T_{adj}^+ = \BA^r$ to obtain the desired local model for $\VinBun_G^{princ}$. We refer to \cite{Sch2} for a more detailed treatment and proofs.

\medskip

\sssec{The open Bruhat locus}
We define \textit{the open Bruhat locus} $\Vin_G^{Bruhat}$ in $\Vin_G$ as the open subvariety obtained by acting by the subgroup $B^- \times N \subset G \times G$ on the section
$$\Fs: \ T_{adj}^+ \ \longto \ \Vin_G \, ,$$
i.e., we define $\Vin_G^{Bruhat}$ as the open image of the map
$$B^- \times N \times T_{adj}^+ \ \ \longto \ \ \Vin_G$$
$$(b,n,t) \ \longmapsto \ (b,n) \cdot \Fs(t) \, .$$
By definition the open Bruhat locus is contained in the non-degenerate locus:
$$\Vin_G^{Bruhat} \ \ \subset \ \ \sideset{_0}{_G}\Vin$$

\medskip

\sssec{GIT-quotients}

We recall the following lemma about the $G \times G$-action on $\Vin_G$:

\begin{lemma}
\label{GIT lemma}
The GIT-quotient
$$\Vin_G \ /\!\!/ \ N \times N^- \ \ \ := \ \ \ \Spec \bigl( k[\Vin_G]^{N \times N^-} \bigr)$$
is naturally isomorphic to $\overline{T} \times T_{adj}^+$. The base change of the resulting map $\Vin_G \longto \Vin_G /\!\!/ N \times N^- = \overline{T} \times T_{adj}^+$ along the inclusion $T \into \overline{T}$ yields a cartesian square
$$\xymatrix@+10pt{
\Vin_G^{Bruhat} \ar[r] \ar[d] & \Vin_G \ar[d]  \\
T \times T_{adj}^+ \ar[r] & \overline{T} \times T_{adj}^+ \\
}$$
in which all arrows are $T$-equivariant. Finally, the left vertical arrow is a $N \times N^-$-torsor; thus we obtain an identification of the stack quotient
$$\Vin_G^{Bruhat}/N \times N^- \ \ \stackrel{\cong}{\longto} \ \ T \times T_{adj}^+ \, .$$
\end{lemma}

\begin{proof}
For the fiberwise (over $T_{adj}^+$) statement, see \cite[3.2.8, 4.1.5]{W1}. The triviality of the family over $T_{adj}^+$ follows from the fact that the Rees filtration (see e.g. \cite[Sec. 5]{GN}) becomes a grading after passing to $N \times N^-$-invariants.
\end{proof}

\medskip

\sssec{The definition of the local models}
Following \cite{Sch2} we now define the local model for the $B$-locus as
$$Y \ \ := \ \ \Maps_{gen}\bigl(X, \, \Vin_G / B \times N^- \ \supset \ \Vin_G^{Bruhat} / B \times N^- \bigr) \, .$$
Just as for $\VinBun_G$, the map $\Vin_G \to T_{adj}^+$ induces a map
$$v: \ Y \ \longto \ T_{adj}^+ \, ,$$
realizing the local model $Y$ as a multi-parameter family over $T_{adj}^+$. As discussed in Subsection \ref{The stratification parametrized by parabolics} for $\VinBun_G$, the map $v$ induces a stratification of the local model $Y$ indexed by parabolic subgroups $P$ of $G$; as we will restrict $Y$ to the principal direction in $T_{adj}^+$, we are again only interested in the $G$-locus $Y_G$ and the $B$-locus $Y_B$. As in the case of $\VinBun_G$, we denote by ${}_0Y$ the defect-free locus, i.e., the locus obtained by requiring that the map from the curve $X$ above factors through the open substack $\sideset{_0}{_G}\Vin / B \times N^-$.

\medskip

\ssec{Basic properties}

\sssec{Structure maps to spaces of divisors}
By Lemma \ref{GIT lemma} above the natural map from the stack quotient to the GIT quotient
$$\Vin_G / N \times N^- \ \ \longto \ \ \Vin_G /\!\!/ N \times N^-$$
induces a map
$$Y \ \ \longto \ \ \Maps_{gen}(X, \overline{T}/T \supset T/T) \ \times \ T_{adj}^+ \, .$$
Composing this map with the projection onto the first factor and using that
$$\Maps_{gen}(X, \overline{T}/T \supset T/T = pt) \ \ = \ \ \bigcup_{\check\theta \in \Lambdach_G^{pos}} X^{\check\theta}$$
we obtain a map
$$Y \ \longto \ \bigcup_{\check\theta \in \Lambdach_G^{pos}} X^{\check\theta} \, .$$
For any element $\check\theta \in \Lambdach_G^{pos}$ we then define $Y^{\check\theta}$ as the inverse image of $X^{\check\theta}$ under this map. We denote the resulting restricted map by
$$\pi: \ Y^{\check\theta} \ \longto \ X^{\check\theta} \, .$$

\medskip

\sssec{Relation to Zastava spaces}
Next let ${}_0Z^{\check\theta}$ denote the defect-free Zastava space for the Borel $B$, introduced and studied in \cite{FM}, \cite{FFKM}, and \cite{BFGM}; its definition is recalled in Subsection \ref{Recollections on Zastava spaces} below.
Our local model $Y^{\check\theta}$ may be viewed as a canonical multi-parameter degeneration of the space ${}_0Z^{\check\theta}$. Indeed, the definition of $Y^{\check\theta}$ implies that the fiber $Y^{\check\theta}|_{c_G}$ of $Y^{\check\theta}$ over the point $c_G \in T_{adj}^+$ is naturally isomorphic to ${}_0Z^{\check\theta}$; however, as in Subsection \ref{Stratification by parabolics} we have the following stronger assertion:

\begin{remark}
\label{G-locus of local model}
The $G$-locus $Y_G^{\check\theta}$ of the local model $Y^{\check\theta}$ forms a canonically trivial fiber bundle over $T_{adj}$:
$$Y_G^{\check\theta} \ \ = \ \ {}_0Z^{\check\theta} \times T_{adj}$$
\end{remark}

\bigskip

\ssec{Recollections on Zastava spaces}
\label{Recollections on Zastava spaces}

\sssec{The definition of Zastava space}

Recall from \cite{FM}, \cite{FFKM}, and \cite{BFGM} that the \textit{Zastava space} $Z$ is defined as
$$Z \ := \ \Maps_{gen}(X, \ (\overline{G/N}) / T \times N^- \ \supset \ pt) \, ,$$
where the dense open point corresponds to the open Bruhat cell $B \cdot N^- \subset G$. We now recall some relevant properties, referring the reader to \cite{FM}, \cite{FFKM}, and \cite{BFGM} for proofs.

\medskip

Similarly to the discussion for our local model $Y$ above, the Zastava space $Z$ decomposes into a disjoint union of spaces $Z^{\check\theta}$ for $\check\theta \in \Lambdach_G^{pos}$, which come equipped with structure maps
$$\pi_Z: \ Z^{\check\theta} \ \longto \ X^{\check\theta} \, .$$
Furthermore, the open subspace
$${}_0Z \ \ := \ \ \Maps_{gen}(X, \ (G/N) / T \times N^- \ \supset \ pt)$$
of $Z$ is smooth.

\medskip

\sssec{Stratification of Zastava spaces}
\label{Zastava stratification}
The Zastava spaces $Z^{\check\theta}$ possess \textit{defect stratifications} similar to the stratification of $\barBun_B$ discussed in Subsection \ref{Drinfeld's relative compactification} above: The action map
$$\overline{T} \times G/N \ \longto \ \overline{G/N}$$
induces locally closed immersions
$$X^{\check\theta'} \times {}_0Z^{\check\theta - \check\theta'} \ \ \longinto \ \ Z^{\check\theta}$$
for any $\check\theta, \check\theta' \in \Lambdach_G^{pos}$ with $\check\theta' \leq \check\theta$. We denote the corresponding locally closed substack by ${}_{\check\theta'}Z^{\check\theta}$. Ranging over all $\check\theta' \in \Lambdach_G^{pos}$ satisfying $0 \leq \check\theta' \leq \check\theta$, the substacks ${}_{\check\theta'}Z^{\check\theta}$ form a stratification of $Z^{\check\theta}$:
$$Z^{\check\theta} \ \ = \ \ \bigcup_{0 \leq \check\theta' \leq \check\theta} \ {}_{\check\theta'}Z^{\check\theta}$$
Finally, the structure map
$$\pi_Z: \ Z^{\check\theta} \ \longto \ X^{\check\theta}$$
admits a natural section
$$X^{\check\theta} \ \longto \ Z^{\check\theta}$$
which maps $X^{\check\theta}$ isomorphically onto the stratum of maximal defect ${}_{\check\theta}Z^{\check\theta}$.

\medskip

\sssec{Relative Zastava spaces}

We will also need a relative version $Z_{\Bun_T}$ of the Zastava space $Z$ introduced above, defined as
$$Z_{\Bun_T} \ := \ \Maps_{gen}(X, \ (\overline{G/N}) / T \times B^- \ \supset \ \cdot/T) \, .$$
The relative Zastava space $Z_{\Bun_T}$ comes equipped with a forgetful map $Z_{\Bun_T} \to \Bun_T$ induced by the composite map
$$(\overline{G/N}) / T \times B^- \ \longto \ \cdot/B^- \ \longto \ \cdot/T \, .$$
Note that the fiber of this forgetful map over the trivial $T$-bundle agrees with the Zastava space $Z$ defined above.
The previous discussion of the Zastava space $Z$ carries over to the relative Zastava space $Z_{\Bun_T}$, with the analogous notation.

\medskip

We can now discuss:

\bigskip

\ssec{Stratification of the $B$-locus of the local models}

The $B$-locus $Y^{\check\theta}_B$ of the local model $Y^{\check\theta}$ admits a stratification analogous to the one of $\VinBun_{G,B}$. To state it, let $\check\theta_1, \check\mu, \check\theta_2, \check\theta \in \Lambdach_G^{pos}$ with $\check\theta_1 + \check\mu + \check\theta_2 = \check\theta$. Then as in Subsection \ref{The defect stratification} above there exist natural compactified strata maps
$$\bar{f}_{\check\theta_1, \check\mu, \check\theta_2}: \ \ Z^{-, \check\theta_1}_{\Bun_T} \underset{\Bun_T}{\times} \Bigl( X^{\check\mu} \times Z^{\check\theta_2} \Bigr) \ \ \ \longto \ \ \ Y^{\check\theta}_B \, ,$$
and the analogous stratification result is:

\medskip

\begin{corollary}
\label{stratification of local models}
The maps $\bar{f}_{\check\theta_1, \check\mu, \check\theta_2}$ are finite, and the restricted maps
$$f_{\check\theta_1, \check\mu, \check\theta_2}: \ \ {}_0Z^{-, \check\theta_1}_{\Bun_T} \underset{\Bun_T}{\times} \Bigl( X^{\check\mu} \times {}_0Z^{\check\theta_2} \Bigr) \ \ \ \longto \ \ \ Y^{\check\theta}_B$$
form isomorphisms onto locally closed substacks
$${}_{\check\theta_1, \check\mu, \check\theta_2}Y^{\check\theta}_B \ \ \longinto \ \ Y^{\check\theta}_B \, .$$
The locally closed substacks ${}_{\check\theta_1, \check\mu, \check\theta_2}Y^{\check\theta}_B$ form a stratification of $Y^{\check\theta}_B$, i.e., on the level of $k$-points the space $Y^{\check\theta}_B$ is equal to the disjoint union
$$Y^{\check\theta}_B \ \ = \ \ \bigcup_{\check\theta_1 + \check\mu + \check\theta_2 \, = \, \check\theta} \ {}_{\check\theta_1, \check\mu, \check\theta_2}Y^{\check\theta}_B \, .$$
\end{corollary}

\medskip

\sssec{Defect and section}
\label{Defect and section}
We use the terms \textit{defect value} and \textit{defect} as in Subsection \ref{The defect stratification} above.
The stratum ${}_{0, \check\theta, 0}Y^{\check\theta}_B$ of maximal defect $\check\mu = \check\theta$ will also be denoted by ${}_{\check\theta}Y^{\check\theta}_B$. By definition we have ${}_{\check\theta}Y^{\check\theta}_B \ \cong \ X^{\check\theta}$. In fact, the structure map
$$\pi: \ Y^{\check\theta} \ \longto \ X^{\check\theta}$$
admits a natural section
$$X^{\check\theta} \ \longto \ Y^{\check\theta}_B \ \longinto \ Y^{\check\theta}$$
which maps $X^{\check\theta}$ isomorphically onto the stratum of maximal defect ${}_{\check\theta}Y^{\check\theta}_B$.

\bigskip

\ssec{Factorization in families}
\label{Factorization in families}

The local models $Y^{\check\theta}$ \textit{factorize in families} over $T_{adj}^+$ in the sense of the following lemma:

\medskip

\begin{lemma}
\label{factorization in families}
Let $\check\theta_1, \check\theta_2 \in \Lambdach_G^{pos}$ and let $\check\theta := \check\theta_1 + \check\theta_2$. Then the addition map of effective divisors
$$X^{\check\theta_1} \stackrel{\circ}{\times} X^{\check\theta_2} \ \longto \ X^{\check\theta}$$
induces the cartesian square
$$\xymatrix@+10pt{
Y^{\check\theta_1} \underset{ \ T_{adj}^+}{\stackrel{\circ}{\times}} Y^{\check\theta_2} \ \ar[rr] \ar[d] & & \ Y^{\check\theta} \ar[d] \\
X^{(\check\theta_1)} \stackrel{\circ}{\times} X^{(\check\theta_2)} \ \ar[rr] & & \ X^{(\check\theta)} \\
}$$
where the top horizontal arrow commutes with the natural maps to $T_{adj}^+$.
\end{lemma}

\medskip

The above lemma implies that the fibers of the map $Y^{\check\theta} \to T_{adj}^+$ are \textit{factorizable} in the usual sense. I.e., for each $t \in T_{adj}^+$ the fiber $Y^{\check\theta}|_t$ fits into the following cartesian square:
$$\xymatrix@+10pt{
Y^{\check\theta_1}|_t \stackrel{\circ}{\times} Y^{\check\theta_2}|_t \ \ar[rr] \ar[d] & & \ Y^{\check\theta}|_t \ar[d] \\
X^{(\check\theta_1)} \stackrel{\circ}{\times} X^{(\check\theta_2)} \ \ar[rr] & & \ X^{(\check\theta)} \\}
$$
In particular, taking $t = c_B \in T_{adj}^+$ we conclude the $B$-locus $Y_B$ is factorizable. Taking $t = c_G \in T_{adj}^+$ we recover the fact that the defect-free Zastava spaces ${}_0Z^{\check\theta}$ are factorizable.

\bigskip\bigskip\bigskip

\section{Proofs II --- The nearby cycles theorem}
\label{Proofs II}

\bigskip

\ssec{Statement of the theorem on the level of local models}

\medskip

\sssec{The principal degeneration of the local model}
\label{The principal degeneration of the local model}
Exactly as for $\VinBun_G$ we denote by $Y^{\check\theta, princ} \to L_B = \BA^1$ the restriction of the local model $Y^{\check\theta} \to T_{adj}^+$ to the principal line $L_B = \BA^1$. The spaces $Y^{\check\theta, princ}$ form local models for the principal degeneration $\VinBun_G^{princ}$ in the sense of \cite[Sec. 3]{BFGM}; they are related to the space $\VinBun_G^{princ}$ in the exact same way in which the Zastava spaces $Z^{\check\theta}$ defined above are related to Drinfeld's relative compactification $\barBun_B$. Like $\VinBun_G^{princ}$, the space $Y^{\check\theta, princ}$ consists of only the $G$-locus $Y_G^{\check\theta, princ} = {}_0Z^{\check\theta} \times (L_B \setminus \{ 0 \})$ and the $B$-locus $Y_B^{\check\theta, princ} = Y_B^{\check\theta}$.
By the exact same argument as in \cite[Sec. 3]{BFGM}, \cite[Sec. 4.3]{BG2}, and also \cite[Sec. 4]{Sch1}, it suffices to prove Theorem \ref{main theorem for nearby cycles} above on the level of the local models $Y^{\check\theta, princ}$. In this subsection we restate the local version of Theorem \ref{main theorem for nearby cycles} above for the convenience of the reader, in Theorem \ref{main theorem for local models} below. The remainder of this section is then devoted to the proof of Theorem \ref{main theorem for local models}.

\medskip

Before restating the main theorem, we note that Lemma \ref{factorization in families} above implies that the principal local model $Y^{\check\theta, princ}$ factorizes in families over the line $L_B = \BA^1$ in the sense that the addition map of effective divisors
$$X^{\check\theta_1} \stackrel{\circ}{\times} X^{\check\theta_2} \ \longto \ X^{\check\theta}$$
induces the cartesian square
$$\xymatrix@+10pt{
Y^{\check\theta_1, princ} \underset{ \ \BA^1}{\stackrel{\circ}{\times}} Y^{\check\theta_2, princ} \ \ar[rr] \ar[d] & & \ Y^{\check\theta, princ} \ar[d] \\
X^{(\check\theta_1)} \stackrel{\circ}{\times} X^{(\check\theta_2)} \ \ar[rr] & & \ X^{(\check\theta)} \\
}$$
where the top horizontal arrow commutes with the natural maps to $\BA^1$.

\medskip

\sssec{The local theorem}
To state the local version of Theorem \ref{main theorem for nearby cycles}, let $\check\theta_1, \check\mu, \check\theta_2, \check\theta \in \Lambdach_G^{pos}$ with $\check\theta_1 + \check\mu + \check\theta_2 = \check\theta$.
We denote by $\IC_{Y^{\check\theta, princ}_G}$ the IC-sheaf of the $G$-locus of $Y^{\check\theta, princ}$. This IC-sheaf is a constant sheaf shifted and twisted according to our conventions in Subsection \ref{Conventions and notation} above. Furthermore, we denote by
$$\IC_{Z^{-, \check\theta_1}_{\Bun_T}} \underset{ \ \Bun_T}{\boxtimes} \Bigl( \CF_{\check\mu} \ \boxtimes \ \IC_{Z^{\check\theta_2}} \Bigr)$$
the $*$-restriction of the external product
$$\IC_{Z^{-, \check\theta_1}_{\Bun_T}} \boxtimes \, \Bigl( \CF_{\check\mu} \ \boxtimes \ \IC_{Z^{\check\theta_2}} \Bigr)$$
from the product space to the fiber product
$$Z^{-, \check\theta_1}_{\Bun_T} \underset{\Bun_T}{\times} \Bigl( X^{\check\mu} \times Z^{\check\theta_2} \Bigr) \, ,$$
shifted by $[- \dim \Bun_T]$ and twisted by $(- \tfrac{\dim \Bun_T}{2})$. The statement then is:

\bigskip

\begin{theorem}
\label{main theorem for local models}
There exists an isomorphism of perverse sheaves
$$\gr \, \Psi (\IC_{Y^{\check\theta, princ}_G}) \ \ \cong \ \bigoplus_{\check\theta_1 + \check\mu + \check\theta_2 = \check\theta} \bar{f}_{\check\theta_1, \check\mu, \check\theta_2,*} \ \Bigl( \IC_{Z^{-, \check\theta_1}_{\Bun_T}} \underset{ \ \Bun_T}{\boxtimes} \Bigl( \CF_{\check\mu} \ \boxtimes \ \IC_{Z^{\check\theta_2}} \Bigr) \Bigr)$$
which identifies the action of the Lefschetz-$\sl_2$ on the right hand side with the monodromy action on the left hand side.
\end{theorem}

\bigskip

Now to prove Theorem \ref{main theorem for nearby cycles} above, it is sufficient to prove Theorem \ref{main theorem for local models} for all $\check\theta \in \check\Lambda_G^{pos}$; this follows by a standard argument usually referred to as the ``interplay principle''. The interplay principle is for example carried out in \cite[Sec. 3, 8.1]{BFGM}, or in \cite[Sec. 4.3, 4.6, 4.7]{BG2}; there the analogous interplay between Drinfeld's compactification $\barBun_B$ and the Zastava spaces is used. In the case of the Drinfeld-Lafforgue-Vinberg degeneration $\VinBun_G$ and the spaces $Y^{\check\theta}$, the interplay principle works completely analogously: For $G=\SL_2$, it is spelled out in our earlier paper \cite{Sch1}, for example in the proof of Lemma 6.2.4 of \cite{Sch1}; the same proof applies verbatim in the case of an arbitrary reductive group $G$ (replacing non-negative integers $n$ by positive coweights $\check\theta \in \check\Lambda_G^{pos}$).

\bigskip

\ssec{Factorization and monodromy}

We begin by showing that the nearby cycles factorize in a manner compatible with the monodromy action:

\medskip

\begin{proposition}
\label{factorization of gr Psi}
Let $\check\theta, \check\theta_1, \check\theta_2 \in \Lambdach_G^{pos}$ with $\check\theta_1 + \check\theta_2 = \check\theta$. Then on the disjoint locus $Y^{\check\theta_1, princ}_B \overset{\circ}{\times} Y^{\check\theta_2, princ}_B$ there exists a canonical isomorphism
$$\gr \, \Psi(\IC_{Y^{\check\theta, princ}_G}) \Big|^*_{Y^{\check\theta_1, princ}_B \overset{\circ}{\times} Y^{\check\theta_2, princ}_B} \ \ = \ \ \gr \, \Psi(\IC_{Y^{\check\theta_1, princ}_G}) \, \overset{\circ}{\boxtimes} \, \gr \, \Psi(\IC_{Y^{\check\theta_2, princ}_G})$$
which respects the action of the Lefschetz-$\sl_2$ on both sides.
Here the left hand side denotes the $*$-pullback of $\gr \, \Psi(\IC_{Y^{\check\theta, princ}_G})$ along the \'etale factorization map
$$Y^{\check\theta_1, princ}_B \overset{\circ}{\times} Y^{\check\theta_2, princ}_B \ \longto \ Y^{\check\theta, princ}_B \, .$$
\end{proposition}

\medskip

\begin{proof}
We first recall how the nearby cycles functor behaves with respect to fiber products. To do so, let $U \to \BA^1$ and $U' \to \BA^1$ be two stacks or schemes over $\BA^1$, let $F$ and $F'$ be perverse sheaves on $U|_{\BA^1 \setminus \{0\}}$ and $U'|_{\BA^1 \setminus \{0\}}$, and denote
$$F \underset{\ \BA^1}{\boxtimes} F' \ := \ \bigl( F \boxtimes F' \bigr) \big|^*_{U \underset{ \, \BA^1}{\times} U'}[-1](-\tfrac{1}{2}) \, .$$
Denote by $N$ and $N'$ the monodromy operators of $\Psi(F)$ and $\Psi(F')$.
Assume finally that the full nearby cycles $\Psi_{full}(F)$, $\Psi_{full}(F')$, and $\Psi_{full}(F \underset{\ \BA^1}{\boxtimes} F')$ are unipotent, i.e., that $\Psi_{full}(F) = \Psi(F)$, $\Psi_{full}(F') = \Psi(F')$, and $\Psi_{full}(F \underset{\ \BA^1}{\boxtimes} F') = \Psi(F \underset{\ \BA^1}{\boxtimes} F')$ in our notation.
Then by \cite[Sec. 5]{BB} there exists a canonical isomorphism
$$\Psi(F \underset{\ \BA^1}{\boxtimes} F') \ \ = \ \ \Psi(F) \boxtimes \Psi(F')$$
on the product $Y|_{\{0\}} \times Y'|_{\{0\}}$ under which the action of the monodromy operator on the left hand side corresponds to the action of
$$N \boxtimes \id \ + \ \id \boxtimes N'$$
on the right hand side. In particular
$$\gr \, \Psi(F \underset{\ \BA^1}{\boxtimes} F') \ \ = \ \ \gr \, \Psi(F) \, \boxtimes \, \gr \, \Psi(F')$$
as representations of the Lefschetz-$\sl_2$.
The assertion of the proposition now follows by applying this fact to the cartesian diagram in Subsection \ref{The principal degeneration of the local model} above. Here we use the fact that the factorization map
$$Y^{\check\theta_1, princ} \underset{ \ \BA^1}{\stackrel{\circ}{\times}} Y^{\check\theta_2, princ} \ \ \longto \ \ Y^{\check\theta, princ}$$
is \'etale and that all three nearby cycles sheaves appearing are unipotent; the former is clear from the cartesian diagram and the latter was shown in \cite[Lemma 8.0.4]{Sch2}.
\end{proof}

\bigskip

\ssec{Stalks of nearby cycles}
\label{Stalks of nearby cycles}

Next we recall a description of the $!$-stalks of the nearby cycles $\Psi (\IC_{Y^{\check\theta, princ}_G})$ from \cite{Sch2}, as well as a description of the IC-stalks of the Zastava spaces from \cite{BFGM}.

\medskip

\sssec{The complex $\widetilde \Omega$}
\label{The complex tilde Omega}
Recall that, given a positive coweight $\check\theta \in \Lambdach_G^{pos}$, we denote by ${}_0Z^{\check\theta}$ the defect-free Zastava space introduced in Subsection \ref{Recollections on Zastava spaces} above, and by $\pi_Z: {}_0Z^{\check\theta} \to X^{\check\theta}$ its structure map. We then define the complex $\widetilde \Omega^{\check\theta}$ on $X^{\check\theta}$ as
$$\widetilde\Omega^{\check\theta} \ \ := \ \ \pi_{Z,!} \, \bigl( \IC_{{}_0Z^{\check\theta}} \bigr) \, .$$

\medskip

\sssec{Stalks of nearby cycles}
In \cite[Thm. 7.1.2]{Sch2} it was shown:

\medskip

\begin{proposition}
\label{stalks}
The $!$-restriction of $\Psi (\IC_{Y^{\check\theta, princ}_G})$ to the stratum of maximal defect
$$X^{\check\theta} \, = \, {}_{\check\theta}Y^{\check\theta, princ}_B \ \ \ \longinto \ \ \ Y^{\check\theta, princ}_B$$
is equal to the complex $\widetilde\Omega^{\check\theta}$.
\end{proposition}

\medskip

\sssec{The complexes $\Omega^{\check\theta}$ and $U^{\check\theta}$ from \cite{BG2}}
In \cite{BG1}, \cite{BG2}, and \cite{BFGM} certain complexes $\Omega^{\check\theta}$ and $U^{\check\theta}$ on $X^{\check\theta}$ were introduced and studied. In the present article we only recall the description of the complex $\Omega^{\check\theta}$ on the level of the Grothendieck group from \cite{BG2}, and a combinatorial description of the complex $U^{\check\theta}$ from \cite{BFGM}. We refer the interested reader to the above sources for the definitions of these complexes and for how they arise in the geometric Langlands program; in the present work they will only appear in the following two ways: First, a description of the complex $\widetilde \Omega^{\check\theta}$ in the Grothendieck group in terms of $\Omega^{\check\theta}$ and $U^{\check\theta}$ can be extracted from \cite{BG2}, as we explain below, and in total we hence obtain a formula for $\widetilde \Omega^{\check\theta}$ in the Grothendieck group which will be used in our proof of Theorem \ref{main theorem for local models} above. Second, the stalks of the IC-sheaf of the Zastava spaces can be described in terms of the complex $U^{\check\theta}$; we will again only need the combinatorial description of $U^{\check\theta}$ mentioned above and given below.

\medskip

To state the descriptions of $\Omega^{\check\theta}$ and $U^{\check\theta}$ we use the notation from Subsection \ref{PLO} above. In particular we will invoke, for any Kostant partition
$$\CK: \ \ \thetacheck \ = \ \sum_{\betacheck \in \check R^+} n_{\betacheck} \betacheck$$
of a positive coweight $\thetacheck \in \Lambdach_G^{pos}$, the partially symmetrized power
$$X^{\CK} \ \ := \ \ \prod_{\betacheck \in \check R^+} X^{(n_{\betacheck})}$$
and the natural map
$$i_{\CK}: \ \ X^{\CK} \ \longto \ X^{\thetacheck}$$
from Subsection \ref{PLO} above.
The descriptions from \cite[Sec. 3.3]{BG2} and from \cite[Thm. 4.5]{BFGM} then are:

\medskip

\begin{lemma}
\label{Omega and U}
\begin{itemize}
\item[]
\item[]
\item[(a)] The complex $U^{\check\theta}$ on $X^{\check\theta}$ decomposes as a direct sum
$$U^{\check\theta} \ \ = \ \ \bigoplus_{\CK \, \in \, \Kost(\check\theta)} i_{\CK,*} \, \Qellbar_{X^{\CK}}[0](0) \, .$$
\item[]
\item[(b)] In the Grothendieck group on $X^{\check\theta}$ the complex $\Omega^{\check\theta}$ agrees with the direct sum
$$\Omega^{\check\theta} \ \ = \ \ \bigoplus_{\CK \, \in \, \Kost(\check\theta)} i_{\CK,*} \, \bigl( \, \underset{\check\beta}{\boxtimes} \, \Lambda^{(n_{\check\beta})}(\Qellbar_X) \, [n_{\check\beta}](n_{\check\beta})\bigr) \, .$$
\end{itemize}
\end{lemma}

\bigskip

\sssec{Description of $\widetilde \Omega^{\check\theta}$ in the Grothendieck group}
We can now recall the aforementioned description of the complex $\widetilde \Omega^{\check\theta}$ in the Grothendieck group, which follows from Corollary 4.5 of \cite{BG2}:

\medskip

\begin{lemma}
\label{tilde Omega lemma}
In the Grothendieck group on $X^{\check\theta}$ we have:
$$\widetilde \Omega^{\check\theta} \ \ = \ \ \sum_{\check\theta_1 + \check\theta_2 = \check\theta} add_* \Bigl( \Omega^{\check\theta_1} \; \boxtimes \ U^{\check\theta_2} \Bigr)$$
Here the sum runs over all pairs of positive coweights $(\check\theta_1, \check\theta_2)$ satisfying $\check\theta_1 + \check\theta_2 = \check\theta$.
\end{lemma}

\medskip

\sssec{IC-stalks of Zastava space}

We will also need the following result about the $!$-stalks of the IC-sheaf of the Zastava space $Z^{\check\theta}$, established in \cite[Sec. 5]{BFGM}:

\medskip

\begin{lemma}
\label{IC of Zastava}
The $!$-restriction of the IC-sheaf $\IC_{Z^{\check\theta}}$ to the stratum of maximal defect ${}_{\check\theta}Z^{\check\theta} = X^{\check\theta}$ is isomorphic to the complex $U^{\check\theta}$.
\end{lemma}

\bigskip

\ssec{Reduction to maximal defect locus}
\label{Reduction to maximal defect locus}

For a positive coweight $\check\theta = \sum_{i \in \CI} n_i \check\alpha_i \in \Lambdach_G^{pos}$ we define the \textit{length} of $\check\theta$ as the integer
$$|\check\theta| \ := \ \sum_{i \in \CI} n_i \, .$$
We will prove Theorem \ref{main theorem for local models} by induction of the length of the positive coweight $\check\theta$ appearing in its formulation. We now begin with proving the induction step. Thus we want to show that Theorem \ref{main theorem for local models} holds for the positive coweight $\check\theta$, and may assume that it holds for all positive coweights of smaller length. In the current subsection we use the induction hypothesis to reduce the assertion of Theorem \ref{main theorem for local models} to the existence of an isomorphism of complexes on the stratum of maximal defect $X^{\check\theta} = {}_{\check\theta}Y^{\check\theta,princ}_B$.

\sssec{Separation according to loci of support}

We first break up the existence of the isomorphism asserted in the theorem into two parts. To do so, we abbreviate
$$R_{\check\theta} \ \ := \ \ \bigoplus_{\check\theta_1 + \check\mu + \check\theta_2 = \check\theta} \bar{f}_{\check\theta_1, \check\mu, \check\theta_2,*} \ \Bigl( \IC_{Z^{-, \check\theta_1}_{\Bun_T}} \underset{ \ \Bun_T}{\boxtimes} \Bigl( \CF_{\check\mu} \ \boxtimes \ \IC_{Z^{\check\theta_2}} \Bigr) \Bigr)$$
and make the following basic observation:

\medskip

\begin{lemma}
\label{semisimplicity}
The perverse sheaf $R_{\check\theta}$ is semisimple. The perverse sheaf $\gr \, \Psi (\IC_{Y^{\check\theta, princ}_G})$ becomes semisimple after forgetting its Weil structure.
\end{lemma}

\medskip

\begin{proof}
For the associated graded $\gr \, \Psi (\IC_{Y^{\check\theta, princ}_G})$ this is a consequence of Gabber's theorem, Proposition \ref{Gabber's theorem} above, together with the decomposition theorem from \cite{BBD} for pure perverse sheaves.
For the perverse sheaf $R_{\check\theta}$ the finiteness of the compactified maps $\bar{f}_{\check\theta_1, \check\mu, \check\theta_2}$ and the decomposition theorem from \cite{BBD} together reduce the assertion of semisimplicity to that of the perverse sheaves $\CF_{\check\mu}$; for the latter it follows from the semisimplicity of the Picard-Lefschetz oscillators discussed in Subsection \ref{PLOs for SL_2} above and the finiteness of the addition maps
$$i_{\CK}: \ \ X^{\CK} \ \longto \ X^{\thetacheck} \, .$$
\end{proof}

\medskip

By Lemma \ref{semisimplicity} above we may split each of the two perverse sheaves whose semisimplicity it asserts into two summands
$$\gr \, \Psi (\IC_{Y^{\check\theta, princ}_G}) = \Bigl( \gr \, \Psi (\IC_{Y^{\check\theta, princ}_G}) \Bigr)_{\text{on} \, {}_{\check\theta}Y^{\check\theta, princ}_B} \bigoplus \Bigl( \gr \, \Psi (\IC_{Y^{\check\theta, princ}_G}) \Bigr)_{\text{not} \, \text{on} \, {}_{\check\theta}Y^{\check\theta, princ}_B}$$
$$R_{\check\theta} \ \ = \ \ \bigl( R_{\check\theta} \bigr)_{\text{on} \, {}_{\check\theta}Y^{\check\theta, princ}_B} \bigoplus \bigl( R_{\check\theta} \bigr)_{\text{not} \, \text{on} \, {}_{\check\theta}Y^{\check\theta, princ}_B}$$
where all simple constituents of the first summand are supported on the locus of maximal defect $X^{\check\theta} = {}_{\check\theta}Y^{\check\theta,princ}_B$ and where all simple constituents of the second summand are not supported on this locus.
By construction these direct sum decompositions are compatible with the action of the Lefschetz-$\sl_2$. We will prove the induction step by separately constructing two isomorphisms:

\medskip

(A) On the locus of maximal defect:
$$\Bigl( \gr \, \Psi (\IC_{Y^{\check\theta, princ}_G}) \Bigr)_{\text{on} \, {}_{\check\theta}Y^{\check\theta, princ}_B} \ \cong \ \bigl( R_{\check\theta} \bigr)_{\text{on} \, {}_{\check\theta}Y^{\check\theta, princ}_B}$$

(B) Away from the locus of maximal defect:
$$\Bigl( \gr \, \Psi (\IC_{Y^{\check\theta, princ}_G}) \Bigr)_{\text{not} \, \text{on} \, {}_{\check\theta}Y^{\check\theta, princ}_B} \ \cong \ \bigl( R_{\check\theta} \bigr)_{\text{not} \, \text{on} \, {}_{\check\theta}Y^{\check\theta, princ}_B}$$

\medskip

\noindent Both isomorphisms will respect the action of the Lefschetz-$\sl_2$. The existence of the isomorphism away from the locus of maximal defect follows easily from the induction hypothesis; the actual work goes in to the existence of the isomorphism on the locus of maximal defect.

\medskip

\sssec{The isomorphism (B)}
\label{The isomorphism (B)}

We now record:

\medskip

\begin{lemma}
There exists an isomorphism of perverse sheaves
$$\Bigl( \gr \, \Psi (\IC_{Y^{\check\theta, princ}_G}) \Bigr)_{\text{not} \, \text{on} \, {}_{\check\theta}Y^{\check\theta, princ}_B} \ \cong \ \bigl( R_{\check\theta} \bigr)_{\text{not} \, \text{on} \, {}_{\check\theta}Y^{\check\theta, princ}_B}$$
which is compatible with the action of the Lefschetz-$\sl_2$.
\end{lemma}

\medskip

\begin{proof}
As in Subsection \ref{The principal degeneration of the local model} above, the induction hypothesis implies the validity of Theorem \ref{main theorem for nearby cycles} after restriction to the locus of defect $< \check\theta$. This in turn implies the validity of Theorem \ref{main theorem for local models} after restriction to the open subscheme ${}_{< \check\theta}Y^{\check\theta, princ}_B$ of $Y^{\check\theta, princ}_B$ consisting of all strata of defect $< \check\theta$, giving rise to an isomorphism of perverse sheaves
$$\gr \, \Psi (\IC_{Y^{\check\theta, princ}_G}) \Big|^*_{{}_{< \check\theta}Y^{\check\theta, princ}_B} \ \cong \bigoplus_{\check\theta_1 + \check\mu + \check\theta_2 = \check\theta} \bar{f}_{\check\theta_1, \check\mu, \check\theta_2,*} \ \Bigl( \IC_{Z^{-, \check\theta_1}_{\Bun_T}} \underset{ \ \Bun_T}{\boxtimes} \Bigl( \CF_{\check\mu} \ \boxtimes \ \IC_{Z^{\check\theta_2}} \Bigr) \Bigr) \Big|^*_{{}_{< \check\theta}Y^{\check\theta, princ}_B}$$
which is compatible with the action of the Lefschetz-$\sl_2$.
Since by definition none of the simple constituents of the perverse sheaf
$$\Bigl( \gr \, \Psi (\IC_{Y^{\check\theta, princ}_G}) \Bigr)_{\text{not} \, \text{on} \, {}_{\check\theta}Y^{\check\theta, princ}_B}$$
are supported on the complement of ${}_{< \check\theta}Y^{\check\theta, princ}_B$, this perverse sheaf must in fact be equal to the intermediate extension of its restriction to ${}_{< \check\theta}Y^{\check\theta, princ}_B$. Applying the intermediate extension functor to the above isomorphism between the restricted perverse sheaves thus yields the desired isomorphism (B) above.
\end{proof}

\bigskip

\ssec{The maximal defect locus}
\label{The maximal defect locus}

\sssec{Notation}
In the present subsection we construct the isomorphism (A), modulo the most involved part of the construction, which is dealt with in Subsection \ref{The diagonal locus} below. To simplify the notation, we denote
$$H_{\check\theta} \ = \ \Bigl( \gr \, \Psi (\IC_{Y^{\check\theta, princ}_G}) \Bigr)_{\text{on} \, {}_{\check\theta}Y^{\check\theta, princ}_B}$$
and observe that
$$\bigl( R_{\check\theta} \bigr)_{\text{on} \, {}_{\check\theta}Y^{\check\theta, princ}_B} \ = \ \CF_{\check\theta} \, ,$$
and then have to provide the desired isomorphism (A) of perverse sheaves
$$H_{\check\theta} \ \cong \ \CF_{\check\theta}$$
on the locus of maximal defect ${}_{\check\theta}Y^{\check\theta, princ}_B$. Recall from Subsection \ref{Defect and section} above that the locus of maximal defect ${}_{\check\theta}Y^{\check\theta, princ}_B$ is canonically identified with the space of divisors $X^{\check\theta}$ via the section 
$$X^{\check\theta} \ \stackrel{\cong}{\longto} \ {}_{\check\theta}Y^{\check\theta, princ}_B \ \longinto \ Y^{\check\theta, princ}$$
of the structure map $Y^{\check\theta, princ} \to X^{\check\theta}$.
For the remainder of the present section we will identify ${}_{\check\theta}Y^{\check\theta, princ}_B$ and $X^{\check\theta}$ without further mention.

\medskip

\sssec{Construction of the isomorphism (A)}
\label{Construction of the isomorphism (A)}
Consider next the diagonal locus $\Delta_X$ of $X^{\check\theta}$, i.e., the closed subvariety
$$\Delta_X: \ X \ \longinto X^{\check\theta}$$
$$x \ \longmapsto \ \check\theta x \, $$
We will construct the isomorphism (A) in two parts: One part on and one part away from the diagonal locus. To do so, we again use Lemma \ref{semisimplicity} above to split the perverse sheaves $H_{\check\theta}$ and $\CF_{\check\theta}$ into summands
$$H_{\check\theta} \ = \ \bigl( H_{\check\theta} \bigr)_{\text{on} \, \Delta_X} \bigoplus \bigl( H_{\check\theta} \bigr)_{\text{not} \, \text{on} \, \Delta_X}$$
$$\CF_{\check\theta} \ = \ \bigl( \CF_{\check\theta} \bigr)_{\text{on} \, \Delta_X} \bigoplus \bigl( \CF_{\check\theta} \bigr)_{\text{not} \, \text{on} \, \Delta_X}$$
where all simple constituents of the first summand are supported on $\Delta_X$ and where all simple constituents of the second summand are not supported on $\Delta_X$.

\medskip

Observe next that the collections of perverse sheaves $H_{\check\theta}$ and $\CF_{\check\theta}$ both admit natural factorization structures which respect the action of the Lefschetz-$\sl_2$. For $\CF_{\check\theta}$ this follows from the fact that the Picard-Lefschetz oscillators have this property, as is explained in Subsection \ref{PLOs for SL_2} above.
For $H_{\check\theta}$ this follows from Proposition \ref{factorization of gr Psi} together with the fact that the locus of maximal defect itself ``factorizes'' in the sense that the following diagram (where $\check\theta_1 + \check\theta_2 = \check\theta$) is cartesian:
$$\xymatrix@+10pt{
{}_{\check\theta_1}Y^{\check\theta_1, princ}_B \stackrel{\circ}{\times} {}_{\check\theta_2}Y^{\check\theta_2, princ}_B \ar[r] \ar[d]   &   {}_{\check\theta}Y^{\check\theta, princ}_B \ar[d]           \\
Y^{\check\theta_1, princ}_B \stackrel{\circ}{\times} Y^{\check\theta_2, princ}_B \ar[r]   &   Y^{\check\theta, princ}_B
}$$

\medskip

The induction hypothesis, which assures that we already have the desired $\sl_2$-equivariant isomorphisms $H_{\check\theta'} \cong \CF_{\check\theta'}$ for all $\check\theta' < \check\theta$, together with the factorization of the perverse sheaves $H_{\check\theta}$ and $\CF_{\check\theta}$ then establishes, via the standard factorization argument from \cite[Sec. 5.4]{BFGM} or \cite[Sec. 6.5]{Sch1}, that there exists an $\sl_2$-equivariant isomorphism of perverse sheaves
$$\bigl( H_{\check\theta} \bigr)_{\text{not} \, \text{on} \, \Delta_X} \ \cong \ \bigl( \CF_{\check\theta} \bigr)_{\text{not} \, \text{on} \, \Delta_X}\, .$$
It remains to construct an $\sl_2$-equivariant isomorphism
$$\bigl( H_{\check\theta} \bigr)_{\text{on} \, \Delta_X} \ \cong \ \bigl( \CF_{\check\theta} \bigr)_{\text{on} \, \Delta_X}$$
on the diagonal
$$\Delta_X: \ X \ \longinto X^{\check\theta} \, .$$
As this part forms the core of the present article and contains most of the content and effort, we prove it separately:

\bigskip

\ssec{The diagonal locus}
\label{The diagonal locus}

We finally come to the key calculation of the present article:

\medskip

\begin{lemma}
\label{lemma on the diagonal}
There exists an isomorphism of perverse sheaves
$$\bigl( H_{\check\theta} \bigr)_{\textnormal{on} \, \Delta_X} \ \cong \ \bigl( \CF_{\check\theta} \bigr)_{\textnormal{on} \, \Delta_X}$$
which is compatible with the action of the Lefschetz-$\sl_2$.
\end{lemma}

\medskip

\begin{proof}
Directly from the definitions one sees that the semisimple perverse sheaf $\CF_{\check\theta}$ admits simple summands supported on the main diagonal $\Delta_X: X \into X^{\check\theta}$ only if $\check\theta$ is a coroot of $G$; indeed, the summand $i_{\CK, *} \, \CP_{\CK}$ of $\CF_{\check\theta}$ is only supported on $\Delta_X$ if the Kostant partition $\CK$ is of length $1$, which forces $\check\theta$ to be a coroot.
In this case, the summand supported on the diagonal $\Delta_X$ is precisely the Picard-Lefschetz oscillator
$$\CP_1 = V \otimes \IC_{\Delta_X} = (\Qellbar(\tfrac{1}{2}) \oplus \Qellbar(-\tfrac{1}{2})) \otimes \Qellbar_{\Delta_X}[1](\tfrac{1}{2}) = \Qellbar_{\Delta_X}[1](1) \oplus \Qellbar_{\Delta_X}[1](0) \, .$$

\medskip

Our task is thus to show that the same holds for $\bigl( H_{\check\theta} \bigr)_{\text{on} \, \Delta_X}$. We will do so by computing the image of $\bigl( H_{\check\theta} \bigr)_{\text{on} \, \Delta_X}$ in the Grothendieck group; since $\bigl( H_{\check\theta} \bigr)_{\text{on} \, \Delta_X}$ is a perverse sheaf, we will able to reconstruct it from its image.
To do so, note first that the $!$-restriction of the nearby cycles to the stratum of maximal defect $X^{\check\theta} = {}_{\check\theta}Y^{\check\theta, princ}_B$ satisfies
$$(1) \ \ \Psi (\IC_{Y^{\check\theta, princ}_G}) \Big|^!_{X^{\check\theta}} \ \ \ = \ \ \ H_{\check\theta} \ \ + \ \ \Bigl( \gr \, \Psi (\IC_{Y^{\check\theta, princ}_G}) \Bigr)_{\text{not} \, \text{on} \, {}_{\check\theta}Y^{\check\theta, princ}_B} \Big|^!_{X^{\check\theta}}$$
in the Grothendieck group on $X^{\check\theta}$ as we do not need to distinguish between the nearby cycles and its associated graded and since $H_{\check\theta}$ is already supported on $X^{\check\theta}$.

\medskip

Next, express the complex
$$\Psi (\IC_{Y^{\check\theta, princ}_G}) \Big|^!_{X^{\check\theta}}$$
as a $\BZ$-linear combination of simple perverse sheaves in the Grothendieck group on $X^{\check\theta}$. Let $S_1^{\check\theta}$ denote the linear combination obtained by dropping all terms appearing in this expression that correspond to simple perverse sheaves not supported on $\Delta_X$. Using the stalk computation from Subsection \ref{Stalks of nearby cycles} above we will compute $S_1^{\check\theta}$ in Lemma \ref{S_1 lemma} below as
$$S_1^{\check\theta} \ = \ \Qellbar_{\Delta_X}(0) - \Qellbar_{\Delta_X}(1)$$
in the case where $\check\theta$ is a coroot, and as $S_1^{\check\theta} = 0$ otherwise.

\medskip

We then proceed analogously for the term
$$\Bigl( \gr \, \Psi (\IC_{Y^{\check\theta, princ}_G}) \Bigr)_{\text{not} \, \text{on} \, {}_{\check\theta}Y^{\check\theta, princ}_B} \Big|^!_{X^{\check\theta}} \, .$$
Express this complex as a $\BZ$-linear combination of simple perverse sheaves in the Grothendieck group, and let $S_2^{\check\theta}$ denote the linear combination obtained by dropping all terms appearing in this expression that correspond to simple perverse sheaves not supported on $\Delta_X$. Using our previous work from Subsection \ref{The maximal defect locus} above we will show in Lemma \ref{S_2 lemma} below that
$$S_2^{\check\theta} \ = \ \Qellbar_{\Delta_X}(0) + \Qellbar_{\Delta_X}(0)$$
in the case where $\check\theta$ is a coroot, and that $S_2^{\check\theta} = 0$ otherwise.

\medskip

Finally, from formula $(1)$ above we compute the image of $\bigl( H_{\check\theta} \bigr)_{\text{on} \, \Delta_X}$ in the Grothendieck group to be
$$\bigl( H_{\check\theta} \bigr)_{\text{on} \, \Delta_X} \ = \ S_1^{\check\theta} - S_2^{\check\theta} \ = \ - \Qellbar_{\Delta_X}(0) - \Qellbar_{\Delta_X}(1) \, .$$
This forces the perverse sheaf $\bigl( H_{\check\theta} \bigr)_{\text{on} \, \Delta_X}$ to be the desired
$$\Qellbar_{\Delta_X}[1](1) \oplus \Qellbar_{\Delta_X}[1](0) \ = \ (\Qellbar(\tfrac{1}{2}) \oplus \Qellbar(-\tfrac{1}{2})) \otimes \IC_{\Delta_X}\, ,$$
so that $\bigl( H_{\check\theta} \bigr)_{\text{on} \, \Delta_X}$ agrees with the Picard-Lefschetz oscillator $\CP_1$ as a perverse sheaf.
To check that the action of the Lefschetz-$\sl_2$ on $\bigl( H_{\check\theta} \bigr)_{\text{on} \, \Delta_X}$ is the correct one, recall from Lemma \ref{Lefschetz-sl_2} and Proposition \ref{Gabber's theorem} that the weights, as a Weil sheaf, of the vector space $\Qellbar(\tfrac{1}{2}) \oplus \Qellbar(-\tfrac{1}{2})$ appearing above as a tensor factor of $\bigl( H_{\check\theta} \bigr)_{\text{on} \, \Delta_X}$ agree with the Cartan weights as an $\sl_2$-representation.
But the only $\sl_2$-representation with these Cartan weights is the standard representation of $\sl_2$, showing that $\bigl( H_{\check\theta} \bigr)_{\text{on} \, \Delta_X}$ agrees with the Picard-Lefschetz oscillator $\CP_1$ also as a perverse sheaf with an action of the Lefschetz-$\sl_2$.
\end{proof}

\medskip

The above proof is completed by establishing the following two lemmas:

\medskip

\begin{lemma}
\label{S_1 lemma}
If $\check\theta$ is a coroot we have
$$S_1^{\check\theta} \ = \ \Qellbar_{\Delta_X}(0) - \Qellbar_{\Delta_X}(1) \, ;$$
otherwise we have $S_1^{\check\theta} = 0$.
\end{lemma}

\medskip

\begin{proof}
By Proposition \ref{stalks} we have
$$\Psi (\IC_{Y^{\check\theta, princ}_G}) \Big|^!_{X^{\check\theta}} \ = \ \widetilde\Omega^{\check\theta} \, .$$
Using Lemma \ref{tilde Omega lemma} and Lemma \ref{Omega and U} we can express the complex $\widetilde\Omega^{\check\theta}$ in the Grothendieck group of $X^{\check\theta}$ in terms of simple perverse sheaves. In doing so, note that in the formula in Lemma \ref{tilde Omega lemma} only the extreme cases $(\check\theta_1, \check\theta_2) = (\check\theta, 0)$ and $(\check\theta_1, \check\theta_2) = (0, \check\theta)$ contribute to the formula for $S_1^{\check\theta}$. For the resulting two terms $\Omega^{\check\theta}$ and $U^{\check\theta}$ we use the formulas in Lemma \ref{Omega and U}: A contribution to $S_1^{\check\theta}$ only happens for the summands corresponding to Kostant partitions $\CK$ of length $1$, which can only happen when $\check\theta$ is a coroot. In this case $\Omega^{\check\theta}$ contributes $\Qellbar_{\Delta_X}[1](1)$ and $U^{\check\theta}$ contributes $\Qellbar_{\Delta_X}[0](0)$, proving the lemma.
\end{proof}

\medskip

\begin{lemma}
\label{S_2 lemma}
If $\check\theta$ is a coroot we have
$$S_2^{\check\theta} \ = \ \Qellbar_{\Delta_X}(0) + \Qellbar_{\Delta_X}(0) \, ;$$
otherwise we have $S_2^{\check\theta} = 0$.
\end{lemma}

\medskip

\begin{proof}
We will use our previous computation from Subsection \ref{Construction of the isomorphism (A)} above. To compute the contribution to $S_2^{\check\theta}$, we need to compute the $!$-restriction of each summand in $\bigl( \CF_{\check\theta} \bigr)_{\text{not} \, \text{on} \, \Delta_X}$ to the stratum $X^{\check\theta}$. We do so by using the cartesian diagram
$$\xymatrix@+10pt{
X^{\check\theta_1} \times X^{\check\mu} \times X^{\check\theta_2} \, \ar[r] \ar^{\add}[d] & \, Z^{-,\check\theta_1}_{\Bun_{T}} \, \underset{\Bun_T}{\times} \bigl( \, X^{\check\mu} \times Z^{\check\theta_2} \bigr) \ar^{\bar{f}_{\check\theta_1, \check\mu, \check\theta_2}}[d] \\
X^{\check\theta} \ \ar[r] & \ Y^{\check\theta, princ}_B \\
}$$
where the top horizontal arrow is induced by the section map discussed in Subsection \ref{Zastava stratification} above, and the bottom horizontal arrow is the section map from Subsection \ref{Defect and section} above.
By this cartesian diagram we need to study the contribution to $S_2^{\check\theta}$ of the pushforwards
$$add_* \Bigl( \IC_{Z^{-, \check\theta_1}} \big|^!_{X^{\check\theta_1}} \boxtimes  \CF_{\check\mu} \ \boxtimes \ \IC_{Z^{\check\theta_2}} \big|^!_{X^{\check\theta_2}} \Bigr)$$
for the addition maps
$$add: \ X^{\check\theta_1} \times X^{\mu} \times X^{\check\theta_2} \ \longto \ X^{\check\theta} \, .$$
As in the proof of Lemma \ref{S_1 lemma} above, the only contributions to $S_2^{\check\theta}$ are made by the extreme cases where either $\check\theta_1 = \check\theta$ or $\check\theta_2 = \check\theta$.
Using that $\IC_{Z^{\check\theta}} \big|^!_{X^{\check\theta}} = U^{\check\theta}$ by Lemma \ref{IC of Zastava} and the formula for $U^{\check\theta}$ in Lemma \ref{Omega and U}, we see that in those two extreme cases a contribution to $S_2^{\check\theta}$ takes place only if the Kostant partition $\CK$ has length $1$, i.e., if $\check\theta$ is a coroot. In this case, the contribution of the two extreme cases is one copy of $\Qellbar_{\Delta_X}[0](0)$ each, as desired.
\end{proof}

\bigskip

\ssec{The base case of the induction}

In Subsections \ref{Reduction to maximal defect locus} through \ref{The diagonal locus} we have completed the induction step of our proof of Theorem \ref{main theorem for local models}. One can verify that no separate base case is needed for the induction: The argument of the induction step goes through to establish the case where the positive coweight $\check\theta$ has length $1$, i.e., is a simple coroot. For the convenience of the reader we now sketch how to indeed arrive at the case of length $1$ via the induction step; one may then alternatively use the case of length $1$ as the base case of the induction.

\medskip

\begin{lemma}
Theorem \ref{main theorem for local models} holds if $\check\theta$ is of length $1$, i.e., if $\check\theta = \check\alpha_i$ is a simple coroot.
\end{lemma}

\medskip

\begin{proof}
We use the same notation as in the induction step in Subsections \ref{Reduction to maximal defect locus} through \ref{The diagonal locus} above.
We first verify the existence of the isomorphism (B). To do so, note first that now the stratum of defect $\check\theta$ is the only defect stratum; its complement in $Y^{\check\theta, princ}$ is the defect-free locus ${}_0Y^{\check\theta, princ}$, which is smooth over $L_B = \BA^1$ by Subsection \ref{The defect-free locus of VinBun_G} above. This implies that the perverse sheaf
$$\Bigl( \gr \, \Psi (\IC_{Y^{\check\theta, princ}_G}) \Bigr)_{\text{not} \, \text{on} \, {}_{\check\theta}Y^{\check\theta, princ}_B}$$
is simply the IC-sheaf of the entire $B$-locus $Y^{\check\theta, princ}_B$ with the trivial $\sl_2$-action. On the other hand, the finiteness of the compactified maps $\bar f$ shows that the complex
$$\bigl( R_{\check\theta} \bigr)_{\text{not} \, \text{on} \, {}_{\check\theta}Y^{\check\theta, princ}_B}$$
is equal to the direct sum of the IC-sheaf of the stratum ${}_{\check\theta, 0, 0}Y^{\check\theta, princ}_B$ and the IC-sheaf of the stratum ${}_{0, 0, \check\theta}Y^{\check\theta, princ}_B$, and both IC-sheaves are equipped with the trivial $\sl_2$-action. As the closures of these two strata form the irreducible components of the $B$-locus $Y^{\check\theta, princ}_B$, this sum is also equal to the IC-sheaf of the entire $B$-locus $Y^{\check\theta, princ}_B$, completing the proof of the existence of the isomorphism (B).

\medskip

It remains to verify the existence of the isomorphism (A). But since the locus of maximal defect ${}_{\check\theta}Y^{\check\theta}_B$ consist of only the diagonal locus $\Delta_X$, we can directly apply Lemma \ref{lemma on the diagonal}, finishing the proof. Alternatively one can also repeat the proof of Lemma \ref{lemma on the diagonal} in the present special case; in this case the proof is dramatically simpler due to the fact that $\bigl( H_{\check\theta} \bigr)_{\text{on} \, \Delta_X} = H_{\check\theta}$ and the fact that no Kostant partition other than the trivial Kostant partition exists.
\end{proof}

\bigskip
\bigskip
\bigskip
\bigskip
\bigskip
\bigskip

\section{Vinberg fusion and a geometric Hopf algebra structure}
\label{Vinberg fusion and a geometric Hopf algebra structure}

\bigskip

\ssec{Recollections and notation}

\sssec{The diagonal fiber}
Fix a $k$-point $x$ of the curve $X$. We denote by $\BY^{\check\theta}$ the fiber of the map $Y^{\check\theta} \longto X^{\check\theta}$ over the point $\check\theta x \in X^{\check\theta}$, and refer to it as the \textit{diagonal fiber}. We use self-explanatory notation such as $\BY^{\check\theta}_G$, $\BY^{\check\theta}_B$, ${}_0\BY^{\check\theta}$, ${}_{\check\theta'}\BY^{\check\theta}_B$, $\BY^{princ}$ to denote the application of various previously discussed notions to the diagonal fiber. Similarly, we denote by $\BZ^{\check\theta}$ the diagonal fiber of the Zastava spaces $Z^{\check\theta}$, and use the notation ${}_0\BZ^{\check\theta}$, ${}_{\check\theta'} \BZ^{\check\theta}$ analogously.

\medskip

\sssec{Irreducible components of the diagonal fiber of Zastava space}
We remark that the irreducible components of the diagonal fiber ${}_0\BZ^{\check\theta}$ have been linked to the Langlands dual group $\check{G}$ in \cite{FM}, \cite{FFKM}, and \cite{BFGM}, building on \cite{MV}. To review the result, let $\check{\Fn}$ denote the Lie algebra of the unipotent part $\check N$ of the Borel $\check B$ of the Langlands dual group $\check G$ of $G$, and let $U(\check{\Fn})$ denote its universal enveloping algebra. The result then is:

\medskip

\begin{lemma}
\label{Zastava cohomology}
The top compactly supported cohomology group $H^{top}_c({}_0\BZ^{\check\theta})$ is canonically identified with the $\check\theta$-weight space $U(\check{\Fn})[\check\theta]$ of $U(\check{\Fn})$.
\end{lemma}

\bigskip

\ssec{Beilinson-Drinfeld fusion and Vinberg fusion}
\label{Beilinson-Drinfeld fusion and Vinberg fusion}

\sssec{Overview}
The non-degenerate principal direction ${}_0Y^{\check\theta, princ}$ comes equipped with a natural map to $\BA^1$, corresponding to the principal Vinberg direction, and a natural map to the space of divisors $X^{\check\theta}$. Taken together, we obtain a map
$${}_0Y^{\check\theta, princ} \ \longto \ X^{\check\theta} \times \BA^1$$
which will play a key role in what follows. One may view this map as a combination of two types of degenerations that are very different in nature: The first degeneration corresponds to ``collisions'' of divisors in $X^{\check\theta}$, and has been pioneered by Beilinson and Drinfeld (\cite{BD1}, \cite{BD2}). The second degeneration is, to our knowledge, new, and is obtained not by degenerating the divisor on the curve, but rather by ``degenerating the group'' $G$ in the ``Vinberg direction''. The space ${}_0Y^{\check\theta, princ}$ realizes both degenerations simultaneously in one total space. As we will see below, this makes the space ${}_0Y^{\check\theta, princ}$ a geometric incarnation of a Hopf algebra, and the two degenerations together indeed define in a geometric way a Hopf algebra structure on the cohomology of the Zastava spaces, i.e., on the universal enveloping algebra $U(\check{\Fn})$.

\medskip

\sssec{The two-parameter degeneration}
For concreteness, we restrict the family ${}_0Y^{\check\theta, princ} \to X^{\check\theta} \times \BA^1$ further, obtaining a two-parameter degeneration $d$ over the product $X \times \BA^1$ as follows. As above we fix a $k$-point $x$ of $X$ and a non-zero positive coweight $\check\theta \in \Lambdach_G^{pos}$. Furthermore, let $\check\theta_1, \check\theta_2 \in \Lambdach_G^{pos} \setminus \{ 0 \}$ such that $\check\theta_1 + \check\theta_2 = \check\theta$. We then define the family
$$d: \ Q \ \longto X \times \BA^1$$
as the pullback of the family
$${}_0Y^{\check\theta, princ} \ \longto \ X^{\check\theta} \times \BA^1$$
along the map
$$X \ \longinto \ X^{\check\theta}$$
$$y \ \longmapsto \ \check\theta_1 y + \check\theta_2 x \, .$$
To make the notation more intuitive and stress that the degeneration $d$ is obtained from the larger family ${}_0Y^{\check\theta, princ}$, we will often abuse notation and denote the fibers of $d$ by expressions such as $Q|_{(\check\theta x, c)}$ and $Q|_{(\check\theta_1 x + \check\theta_2 y, c)}$.

\medskip

We can already observe:

\medskip

\begin{lemma}
\label{fiber lemma}
\begin{itemize}
\item[]
\item[]
\item[(a)] The fiber $Q|_{(\check\theta x, \, 1)}$ of $d$ over the point $(\check\theta x, 1)$ is naturally identified with ${}_0\BZ^{\check\theta}$.
\item[]
\item[(b)] On the level of the underlying reduced schemes, the fiber $Q|_{(\check\theta x, \, 0)}$ of $d$ over the point $(\check\theta x, 0)$ decomposes into a disjoint union of open and closed components
$$\bigcup_{\check\mu_1 + \check\mu_2 \, = \, \check\theta} \ {}_0\BZ^{\check\mu_1} \times {}_0\BZ^{\check\mu_2}$$
where the union runs over all positive coweights $\check\mu_1, \check\mu_2 \in \Lambdach_G^{pos}$ satisfying that $\check\mu_1 + \check\mu_2 = \check\theta$.
\item[]
\item[(c)] The fiber $Q|_{(\check\theta_1 x + \check\theta_2 y, \, 1)}$ of $d$ over the point $(\check\theta_1 x + \check\theta_2 y, 1)$ is naturally identified with ${}_0\BZ^{\check\theta_1} \times {}_0\BZ^{\check\theta_2}$.
\end{itemize}
\end{lemma}

\medskip

\begin{proof}
Part (a) follows directly from Remark \ref{G-locus of local model}.
Corollary \ref{stratification of local models} immediately shows Part (b).
Finally, Part (c) follows from Remark \ref{G-locus of local model} together with Lemma \ref{factorization in families} and the comment following it.
\end{proof}

\medskip

\sssec{Two one-parameter degenerations}
Next let $d_{comult}$ denote the one-parameter family over the curve $X$ obtained by restricting the family $d$ above along the inclusion
$$X \times \{ 1 \} \ \longinto \ X \times \BA^1 \, .$$
Similarly, let $d_{mult}$ denote the one-parameter family over the affine line $\BA^1$ obtained by restricting the family $d$ above along the inclusion
$$\{ x \} \times \BA^1 \ \longinto \ X \times \BA^1 \, .$$

\medskip

Using Lemma \ref{fiber lemma} above we find:

\medskip

\begin{corollary}
\label{special versus general}
\begin{itemize}
\item[]
\item[]
\item[(a)] The one-parameter family $d_{comult}$ is trivial over $X \setminus \{ x \}$. It deforms the special fiber
$$Q|_{(\check\theta x, \, 1)} \ = \ {}_0\BZ^{\check\theta}$$
to the general fiber
$$Q|_{(\check\theta_1 x + \check\theta_2 y, \, 1)} \ = \ {}_0\BZ^{\check\theta_1} \times {}_0\BZ^{\check\theta_2} \, .$$
\item[]
\item[(b)] The one-parameter family $d_{mult}$ is trivial over $\BA^1 \setminus \{ 0 \}$. It deforms the special fiber $Q|_{(\check\theta x, \, 0)}$, which on the level of reduced schemes agrees with the disjoint union
$$\bigcup_{\check\mu_1 + \check\mu_2 \, = \, \check\theta} \ {}_0\BZ^{\check\mu_1} \times {}_0\BZ^{\check\mu_2} \, ,$$
to the general fiber $Q|_{(\check\theta x, \, 1)} = {}_0\BZ^{\check\theta}$.
\end{itemize}
\end{corollary}

\medskip

\sssec{Remarks}
The family $d_{comult}$ is nothing but the usual factorizable version of the Zastava space corresponding to the degeneration of divisors, and has been studied in \cite{BFGM}. Like a coalgebra structure, this family deforms a single Zastava space into a product of Zastava spaces; it forms an example of the usual Beilinson-Drinfeld fusion from \cite{BD1}, \cite{BD2}.
As discussed before, the family $d_{mult}$ is, to our knowledge, new, and may be considered a \textit{Vinberg degeneration} of the Zastava spaces. Like an algebra structure, this family deforms a product of Zastava spaces into a single Zastava space; we refer to this process as \textit{Vinberg fusion}.
The key phenomenon, to be exploited below, is that both degenerations are compatible in the sense that their are obtained as subfamilies of the larger family
$${}_0Y^{\check\theta, princ} \ \longto \ X^{\check\theta} \times \BA^1 \, .$$

\medskip

\sssec{Cospecialization}
\label{Cospecialization}
Let $\pi: S \to \BA^1$ be a one-parameter family over $\BA^1$, and assume that the family $S$ is trivial over $\BA^1 \setminus \{ 0 \}$, i.e., that there exists an isomorphism
$$S|_{\BA^1 \setminus \{ 0 \}} \ \ \ \cong \ \ \ S|_1 \ \times \ (\BA^1 \setminus \{ 0 \}) \, .$$
Then there exists a canonical \textit{cospecialization map} on compactly supported cohomology
$$H^*_c(S|_0) \ \longto \ H^*_c(S|_1) \, .$$
In sheaf-theoretic language, this map is obtained as the canonical map between stalks
$$F|^*_0 \ \longto \ F|^*_1$$
of the constructible complex $F := \pi_! (\Qellbar)_S$ on $\BA^1$, using that the hypothesis on the map $\pi$ implies that $F$ is constant on $\BA^1 \setminus \{ 0 \}$.

\medskip

The same construction applies, and the same notation will be used, if the one-parameter family is not parametrized by the affine line $\BA^1$ but by an arbitrary smooth curve with a fixed $k$-point $x$, playing the role of $0 \in \BA^1$. In the next subsection we will study the cospecialization maps corresponding to the two one-parameter families $d_{comult}$ and $d_{mult}$ defined above.

\medskip

In the proofs of our statements, we will also need cospecialization maps on higher-dimensional varieties. Thus we recall that, given a stratified variety $S$ and a complex $F$ on $S$ which is constant on the strata, there exists a natural cospecialization map
$$F|^*_s \ \longto \ F|^*_t$$
whenever the stratum containing the point $s$ lies in the closure of the stratum containing the point $t$.

\bigskip

\ssec{Construction of the Hopf algebra}

\sssec{The underlying graded vector space}
Given $\check\theta \in \Lambdach_G^{pos}$ we define a vector space
$$A[\check\theta] \ := \ H^{top}_c({}_0Y^{\check\theta, princ}|_{(\check\theta x, \, 1)}) \, .$$
In particular, by Lemma \ref{fiber lemma} and Lemma \ref{Zastava cohomology} we have canonical identifications
$$A[\check\theta] \ = \ H^{top}_c({}_0\BZ^{\check\theta}) \ = \ U(\check{\Fn})[\check\theta] \, .$$
We define $A$ as the $\Lambdach_G^{pos}$-graded vector space
$$A \ := \ \bigoplus_{\check\theta \in \Lambdach_G^{pos}} A[\check\theta] \, .$$
Thus as $\Lambdach_G^{pos}$-graded vector spaces the space $A$ agrees with the universal enveloping algebra $U(\check{\Fn})$.

\medskip

\sssec{The comultiplication map}
Given positive coweights $\check\theta, \check\theta_1, \check\theta_2 \in \Lambdach_G^{pos}$ with $\check\theta_1 + \check\theta_2 = \check\theta$ we define a map of vector spaces
$$\textit{comult}: \ A[\check\theta] \ \longto \ A[\check\theta_1] \otimes A[\check\theta_2]$$
as the cospecialization map corresponding to the one-parameter degeneration $d_{comult}$ obtained from the space ${}_0Y^{\check\theta, princ}$ in Subsection \ref{Beilinson-Drinfeld fusion and Vinberg fusion} above; by Lemma \ref{Zastava cohomology} and Corollary \ref{special versus general} above, this cospecialization map indeed maps the source $A[\check\theta]$ to the target $A[\check\theta_1] \otimes A[\check\theta_2]$.
Summing over all such triples $(\check\theta, \check\theta_1, \check\theta_2)$ we obtain a map
$$\textit{comult}: \ A \ \longto \ A \otimes A \, .$$

\medskip

\sssec{The multiplication map}
We now make the key definition of the present section: We define what will turn out to be the multiplication map of a Hopf algebra structure on $A$. To do so, let $\check\theta \in \Lambdach_G^{pos}$ as before. Then we define the map
$$\textit{mult}: \ \bigoplus_{\check\theta_1 + \check\theta_2 = \check\theta} A[\check\theta_1] \otimes A[\check\theta_2] \ \longto \ A[\check\theta]$$
as the cospecialization map corresponding to the one-parameter degeneration $d_{mult}$ obtained from the space ${}_0Y^{\check\theta, princ}$ in Subsection \ref{Beilinson-Drinfeld fusion and Vinberg fusion} above; again, Lemma \ref{Zastava cohomology} and Corollary \ref{special versus general} show that this cospecialization map indeed maps the source $\bigoplus_{\check\theta_1 + \check\theta_2 = \check\theta} A[\check\theta_1] \otimes A[\check\theta_2]$ to the target $A[\check\theta]$.
Summing over all $\check\theta \in \Lambdach_G^{pos}$ we obtain the desired map
$$\textit{mult}: \ A \otimes A \ \longto \ A \, .$$

\bigskip

\ssec{Associativity of the multiplication via geometry}
\label{Associativity of the multiplication via geometry}

In this subsection we will show:

\medskip

\begin{proposition}
\label{associativity proposition}
The multiplication map
$$mult: \ A \otimes A \ \longto \ A$$
is associative.
\end{proposition}

\medskip

\sssec{The geometric idea for the proof}

Before proceeding to the actual proof, we describe here its rather basic idea, which is entirely geometric in nature: Recall that the map $mult$ is obtained as the cospecialization map of the family $d_{mult}$ over $\BA^1$. In the proof we will construct a family over the ``square'' $\BA^1 \times \BA^1$. The top compactly supported cohomology of the fibers of this family takes the following shape: Over the point $(1,1)$, or in fact over any point away from the coordinate axes, it is equal to (an appropriate summand of) $A$. Over the coordinate axes, yet away from the origin $(0,0)$, it is equal to (an appropriate summand of) $A \otimes A$. Finally, over the origin $(0,0)$ it is equal to (an appropriate summand of) $A \otimes A \otimes A$.

\medskip

In this geometric setup, one can cospecialize from the point $(0,0)$ to a general point by composing two cospecialization steps: The first step consists of cospecializing from the most special point $(0,0)$ to one of the two coordinate axes, and the second step consists of cospecializing from the chosen coordinate axis to a general point in the interior of the ``square'' $\BA^1 \times \BA^1$. Depending on the choice of axis one obtains two a priori different two-step cospecialization procedures, but their composites agree as both must agree with the direct diagonal cospecialization. We will then realize each of the two two-step cospecialization procedures as a way to multiply three elements in $A$ via the map $mult$, and the agreement between the composites will boil down precisely to the associativity axiom for $mult$.

\medskip

We now carry this idea out:

\medskip

\sssec{Proof of associativity}

\begin{proof}
We construct the desired two-parameter family
$$R \ \longto \ \BA^1 \times \BA^1$$
analogously to the construction in Subsection \ref{Construction of local models} above.
Let $G$ act on the product $\Vin_G \times \Vin_G$ via the anti-diagonal action obtained from the action of $G = \{1\} \times G \into G \times G$ on the first factor and the action of $G = G \times \{1\} \into G \times G$ on the second factor. Denote by $D$ the quotient of $\Vin_G \times \Vin_G$ by this $G$-action; thus $D$ still carries a $G \times G$-action, corresponding to the action of $G \times \{1\}$ on the first copy of $\Vin_G$ and the action of $\{1\} \times G$ on the second copy of $\Vin_G$. The multiplication map $\Vin_G \times \Vin_G \to \Vin_G$ of the Vinberg semigroup descends to a map $D \to \Vin_G$, and we denote by $D^{Bruhat}$ the inverse image of $\Vin_G^{Bruhat}$ under this map. Let $\tilde R$ denote the open substack of the mapping stack
$$\Maps \bigl(X, \, D / B \times N^- \bigr)$$
consisting of those maps $X \to D / B \times N^-$ which send the complement $X \setminus \{x\}$ of the fixed $k$-point $x$ of $X$ to the open substack
$$D^{Bruhat} / B \times N^- \ \longinto \ D / B \times N^- \, .$$
Next, the product map $\Vin_G \times \Vin_G \to T_{adj}^+ \times T_{adj}^+$ descends to a map $D \to T_{adj}^+ \times T_{adj}^+$, which in turn induces a map
$$\tilde R \ \longto \ T_{adj}^+ \times T_{adj}^+ \, .$$
Finally, we obtain the desired two-parameter family
$$R \ \longto \ \BA^1 \times \BA^1$$
by restricting the family $\tilde R$ to the product of the principal directions
$$L_B \times L_B \ = \ \BA^1 \times \BA^1 \ \ \longinto \ \ T_{adj}^+ \times T_{adj}^+ \, .$$

\medskip

As in Subsection \ref{Construction of local models} above the total space $R$ of this family decomposes into a disjoint union of open and closed components
$$R \ = \ \bigcup_{\check\theta \in \Lambdach_G^{pos}} R^{\check\theta} \, .$$
By construction, each two-parameter family $R^{\check\theta} \to \BA^1 \times \BA^1$ is trivial over $\BA^1 \setminus \{0\} \times \BA^1 \setminus \{0\}$; as in Lemma \ref{fiber lemma} (a) above one sees that the fiber over this locus is equal to ${}_0\BZ^{\check\theta}$; hence the top compactly supported cohomology of this fiber is equal to $A[\check\theta]$.

\medskip

Similarly, the restriction of this family to $(\BA^1 \setminus \{0\}) \times \{0\}$ and to $\{0\} \times (\BA^1 \setminus \{0\})$ is trivial; it follows as in Lemma \ref{fiber lemma} (b) above that on the level of reduced schemes, the fiber of either of these two restrictions decomposes into a disjoint union of open and closed components
$$\bigcup_{\check\theta_1 + \check\theta_2 \, = \, \check\theta} \ {}_0\BZ^{\check\theta_1} \times {}_0\BZ^{\check\theta_2}$$
where the union runs over all positive coweights $\check\theta_1, \check\theta_2 \in \Lambdach_G^{pos}$ such that $\check\theta_1 + \check\theta_2 = \check\theta$; hence the top compactly supported cohomology of the fiber of either of these two restrictions is equal to $\bigoplus_{\check\theta_1 + \check\theta_2 = \check\theta} A[\check\theta_1] \otimes A[\check\theta_2]$.

\medskip

Finally, as in Lemma \ref{fiber lemma} (b) above one sees that the fiber of the family $R$ over the origin $(0,0)$ decomposes, again on the level of reduced schemes, into a disjoint union of open and closed components
$$\bigcup_{\check\theta_1 + \check\theta_2 + \check\theta_3 \, = \, \check\theta} \ {}_0\BZ^{\check\theta_1} \times {}_0\BZ^{\check\theta_2} \times {}_0\BZ^{\check\theta_3}$$
where the union runs over all positive coweights $\check\theta_1, \check\theta_2, \check\theta_3 \in \Lambdach_G^{pos}$ such that $\check\theta_1 + \check\theta_2 + \check\theta_3 = \check\theta$.
Here, recalling the description of $\Vin_{G,B}$ from Lemma \ref{Vinberg lemma} above, the ``middle'' Zastava space arises from the fact that the \textit{second} copy of $G$ in the \textit{first} copy of $\Vin_{G,B}$ is identified with the \textit{first} copy of $G$ in the \textit{second} copy of $\Vin_{G,B}$ in the definition of $D$ above.
In particular, the top compactly supported cohomology of this fiber is equal to
$$\bigoplus_{\check\theta_1 + \check\theta_2 + \check\theta_3 = \check\theta} A[\check\theta_1] \otimes A[\check\theta_2] \otimes A[\check\theta_3] \, .$$

\medskip

The origin $(0,0)$ lies in the closures of the loci $(\BA^1 \setminus \{0\}) \times \{0\}$ and $\{0\} \times (\BA^1 \setminus \{0\})$, and those loci in turn lie in the closure of $(\BA^1 \setminus \{0\}) \times (\BA^1 \setminus \{0\})$. Thus by Subsection \ref{Cospecialization} above, we obtain cospecialization maps
$$H^{top}_c(R|_{(0,0)}) \ \longto \ H^{top}_c(R|_{(0,1)}) \ \longto \ H^{top}_c(R|_{(1,1)})$$
and
$$H^{top}_c(R|_{(0,0)}) \ \longto \ H^{top}_c(R|_{(1,0)}) \ \longto \ H^{top}_c(R|_{(1,1)}) \, ,$$
and the two composite maps agree since both agree with the direct cospecialization map $H^{top}_c(R|_{(0,0)}) \to H^{top}_c(R|_{(1,1)})$ resulting from the fact that the origin $(0,0)$ of course also lies in the closure of $(\BA^1 \setminus \{0\}) \times (\BA^1 \setminus \{0\})$. But by construction the first composite map takes the form
$$\xymatrix@+10pt{
\bigoplus_{\check\theta_1 + \check\theta_2 + \check\theta_3 = \check\theta} A[\check\theta_1] \otimes A[\check\theta_2] \otimes A[\check\theta_3] \ar[rr]^{ \ \ \ \ \ \ \ \ \ \ \ \ \ \ \ \ \ \ mult(id \otimes mult)} & & A[\check\theta] \, , \\
}$$
while the second composite map takes the form
$$\xymatrix@+10pt{
\bigoplus_{\check\theta_1 + \check\theta_2 + \check\theta_3 = \check\theta} A[\check\theta_1] \otimes A[\check\theta_2] \otimes A[\check\theta_3] \ar[rr]^{ \ \ \ \ \ \ \ \ \ \ \ \ \ \ \ \ \ \ mult(mult \otimes id)} & & A[\check\theta] \, , \\
}$$
proving the associativity of the multiplication map.
\end{proof}

\bigskip

\ssec{The Hopf algebra axiom via geometry}
In this Subsection we show that the triple $(A, mult, comult)$ indeed forms a Hopf algebra. Since $A$ is a graded algebra and coalgebra over the field $k$ and since $A_0 = k$, it suffices to show that the multiplication and comultiplication structures are compatible; the existence of the antipode is automatic in this setting.
We note that the possibility of relating the maps $mult$ and $comult$ is the key feature of the local model studied in the present article; this observation was in fact the starting point for the study of the Vinberg fusion. More precisely, the two degenerations $d_{mult}$ and $d_{comult}$ giving rise to the maps $mult$ and $comult$ could have been constructed independently of each other, and indeed the degeneration $d_{comult}$ is well-known from \cite{BFGM} --- unlike the degeneration $d_{mult}$, it is not at all related to the Vinberg semigroup. The degeneration $d_{mult}$ is however new, and we will now exploit that both degenerations naturally appear together in the principal direction of the defect-free local model
$${}_0Y^{\check\theta, princ} \ \longto \ X^{\check\theta} \times \BA^1$$
as the two-parameter sub-family
$$d: \ Q \ \longto X \times \BA^1 \, .$$
Concretely, we will use the two-parameter family $d$ to show:

\medskip

\begin{proposition}
The map
$$comult: \ A \ \longto \ A \otimes A$$
is an algebra homomorphism for the algebra structures on $A$ and $A \otimes A$ defined by the map $mult$. Equivalently, the map
$$mult: \ A \otimes A \ \longto \ A$$
is a coalgebra homorphism for the coalgebra structures on $A \otimes A$ and $A$ defined by the map $comult$. Thus the triple $(A, mult, comult)$ indeed forms a Hopf algebra.
\end{proposition}

\medskip

\sssec{The geometric idea for the proof}
As in Subsection \ref{Associativity of the multiplication via geometry} above, we will give a geometric proof; we now explain its idea and only then proceed to the actual proof.

\medskip

The idea of verifying the Hopf algebra axiom is the following: Recall that we have fixed a $k$-point $x$ on the curve $X$. In the geometric setup of the two-parameter degeneration $d$ over the ``square'' $X \times \BA^1$, one can cospecialize from the most special point $(x,0)$ to a general point in the interior of the ``square'' $X \times \BA^1$ by choosing one of two two-step procedures:
The first procedure first cospecializes from $(x,0)$ to the ``axis'' $X \times \{ 0 \}$, and then cospecializes from the ``axis'' $X \times \{ 0 \}$ to the interior of the ``square'' $X \times \BA^1$. The second procedure is analogous, using the other ``axis'' $\{ x \} \times \BA^1$ instead.

\medskip

Depending on the choice of ``axis'' one again obtains two composite cospecialization procedures, and again they must agree since both agree with the direct ``diagonal'' cospecialization. On the level of the top compactly supported cohomology of the fibers, this compatibility yields precisely the commutativity of the relevant diagram. We now carry this idea out:

\medskip

\sssec{Proof of the Hopf algebra axiom}

\begin{proof}
We need to show that the diagram

$$\xymatrix@+10pt{
A \otimes A \otimes A \otimes A \ar[rrr]^{ \ \ \ \ \ (mult \, \otimes \, mult) \, \circ \, (id \, \otimes \, \tau \, \otimes id)} & & & A \otimes A \\
A \otimes A \ar[rrr]^{mult} \ar[u]^{comult \, \otimes \, comult} & & & A \ar[u]^{comult} \\
}$$

\medskip

\noindent commutes, where the map $\tau: A \otimes A \to A \otimes A$ sends $a \otimes b$ to $b \otimes a$.
Following the above outline, we will obtain this diagram via cospecialization in the two-parameter family $d$; as in the proof of Proposition \ref{associativity proposition} above, this immediately implies the required commutativity.

\medskip

The fibers of the family $d: Q \to X \times \BA^1$ were already described in Lemma \ref{fiber lemma}, except for the fiber over the locus $(X \setminus \{ x \}) \times \{ 0 \}$. Before describing the latter, we introduce the following additional subscripts with the goal of aligning the notation with the proof strategy. In the setting of Corollary \ref{special versus general}, we introduce the subscripts ``$(x)$'' and ``$(y)$'' on the right hand side of the identification
$$Q|_{(\check\theta_1 x + \check\theta_2 y, \, 1)} \ = \ {}_0\BZ^{\check\theta_1}_{(x)} \times {}_0\BZ^{\check\theta_2}_{(y)}$$
to indicate that this product decomposition into smaller Zastava spaces is due to the Beilinson-Drinfeld fusion of the points $x$ and $y$ on the curve $X$. Similarly, we introduce the subscripts $(1)$ and $(2)$ in the description, on the level of reduced schemes, of the special fiber $Q|_{(\check\theta x, \, 0)}$ as the disjoint union
$$\bigcup_{\check\mu_1 + \check\mu_2 \, = \, \check\theta} \ {}_0\BZ^{\check\mu_1}_{(1)} \times {}_0\BZ^{\check\mu_2}_{(2)}$$
to indicate that this product decomposition into smaller Zastava spaces is due to the Vinberg fusion.

\medskip

With this notation, the fiber $Q|_{(\check\theta_1 x + \check\theta_2 y, 0)}$ of the family $d$ over a point lying in the locus $(X \setminus \{ x \}) \times \{ 0 \}$ decomposes, on the level of reduced schemes, into a disjoint union
$$\bigcup_{(\check\gamma_1, \check\gamma_2, \check\delta_1, \check\delta_2)} \ {}_0\BZ^{\check\gamma_1}_{(1,x)} \times {}_0\BZ^{\check\gamma_2}_{(1,y)} \times {}_0\BZ^{\check\delta_1}_{(2,x)} \times {}_0\BZ^{\check\delta_2}_{(2,y)}$$
where the union runs over all $\check\gamma_1, \check\gamma_2, \check\delta_1, \check\delta_2 \in \Lambdach_G^{pos}$ such that $\check\gamma_1 + \check\delta_1 = \check\theta_1$ and $\check\gamma_2 + \check\delta_2 = \check\theta_2$. Here the subscripts indicate how the corresponding copies of Zastava space arise in the family $d$ via Beilinson-Drinfeld fusion and via Vinberg fusion from the Zastava spaces occurring in the descriptions of the fibers $Q|_{(\check\theta x, 0)}$ and $Q|_{(\check\theta_1 x + \check\theta_2 y, 1)}$.

\medskip

From this description we see that the cospecialization map from the top compactly supported cohomology of the fiber $Q|_{(\check\theta x, 0)}$ to the top compactly supported cohomology of the fiber $Q|_{(\check\theta_1 x + \check\theta_2 y, 0)}$ realizes the left vertical arrow
$$\xymatrix@+10pt{
\bigoplus_{\check\mu_1 + \check\mu_2 = \check\theta} A[\check\mu_1] \otimes A[\check\mu_2] \ar[rr]^{comult \, \otimes \, comult \ \ \ \ \ \ \ \ \ \ \ \ } & & \bigoplus_{(\check\gamma_1, \check\gamma_2, \check\delta_1, \check\delta_2)} A[\check\gamma_1] \otimes A[\check\gamma_2] \otimes A[\check\delta_1] \otimes A[\check\delta_2]
}$$
of the desired diagram; here the summand on the left hand side corresponding to the pair $(\check\mu_1, \check\mu_2)$ maps to the summands on the right hand side corresponding to those quadruples $(\check\gamma_1, \check\gamma_2, \check\delta_1, \check\delta_2)$ satisfying $\check\gamma_1 + \check\gamma_2 = \check\mu_1$ and $\check\delta_1 + \check\delta_2 = \check\mu_2$.

\medskip

Next note that, as indicated by the subscripts in the product decomposition
$${}_0\BZ^{\check\gamma_1}_{(1,x)} \times {}_0\BZ^{\check\gamma_2}_{(1,y)} \times {}_0\BZ^{\check\delta_1}_{(2,x)} \times {}_0\BZ^{\check\delta_2}_{(2,y)}$$
appearing in the description of the fiber $Q|_{(\check\theta_1 x + \check\theta_2 y,0)}$ above, the first factor of the product $Q|_{(\check\theta_1 x + \check\theta_2 y,1)} = {}_0\BZ^{\check\theta_1}_{(x)} \times {}_0\BZ^{\check\theta_2}_{(y)}$ degenerates to the product of the first and third Zastava factors, while the second factor degenerates to the second and third Zastava factors, under the Vinberg degeneration of the fiber $Q|_{(\check\theta_1 x + \check\theta_2 y, 1)}$ to the fiber $Q|_{(\check\theta_1 x + \check\theta_2 y,0)}$. Hence the cospecialization map from the top compactly supported cohomology of the fiber $Q|_{(\check\theta_1 x + \check\theta_2 y,0)}$ to the top compactly supported cohomology of the fiber $Q|_{(\check\theta_1 x + \check\theta_2 y,1)}$ realizes the top horizontal arrow
$$\xymatrix@+10pt{
\bigoplus_{(\check\gamma_1, \check\gamma_2, \check\delta_1, \check\delta_2)} A[\check\gamma_1] \otimes A[\check\gamma_2] \otimes A[\check\delta_1] \otimes A[\check\delta_2] \ar[rrr]^{ \ \ \ \ \ \ \ \ \ \ \ \ \ \ \ \ \ \ \ (mult \, \otimes \, mult) \, \circ \, (id \, \otimes \, \tau \, \otimes id)} & & & A[\check\theta_1] \otimes A[\check\theta_2] \\
}$$

\medskip

Finally, by definition of the map $mult$ the degeneration of the fiber $Q|_{(\check\theta x, 1)}$ to the fiber $Q|_{(\check\theta x, 0)}$ induces, in the same fashion, the bottom horizontal arrow
$$\xymatrix@+10pt{
\bigoplus_{\check\mu_1 + \check\mu_2 = \check\theta} A[\check\mu_1] \otimes A[\check\mu_2] \ar[r]^{ \ \ \ \ \ \ \ \ \ \ \ \ \ mult} & A[\check\theta] \\
}$$
of the desired diagram. Similarly, by definition of the map $comult$ the degeneration of the fiber $Q|_{(\check\theta_1 x + \check\theta_2 y,1)}$ to the fiber $Q|_{(\check\theta x, 1)}$ induces the right vertical arrow
$$\xymatrix@+10pt{
A[\check\theta] \ar[r]^{comult \ \ \ \ \ \ } & A[\check\theta_1] \otimes A[\check\theta_2] \, .
}$$

\end{proof}

\medskip

\ssec{A question regarding the Hopf algebra structure}
\label{A question regarding the Hopf algebra structure}

We have constructed the comultiplication of our Hopf algebra $A$ via the usual Beilinson-Drinfeld fusion for Zastava spaces, i.e., via the cospecialization map of the one-parameter degeneration $d_{comult}$ defined above. As has already been discussed, this construction is well-known and is unrelated to the Vinberg semigroup. In fact, it is already known from the works \cite{BG2}, \cite{FFKM} that this comultiplication on the vector space $A = U(\check{\Fn})$ agrees with the usual comultiplication on $U(\check{\Fn})$. The datum of the multiplication on $U(\check{\Fn})$ is equivalent to the datum of the Langlands dual Lie bracket and thus arguably much more interesting. One may ask:

\medskip

\begin{question}
Does the multiplication map $\textit{mult}$ of the Hopf algebra structure constructed above agree with the multiplication map of $U(\check{\Fn})$? In other words, is the Hopf algebra $A$ constructed geometrically via the two-parameter degeneration $d$ equal to the universal enveloping algebra $U(\check{\Fn})$? Put informally, can the Langlands dual Lie bracket between two elements of $\check{\Fn}$ be obtained geometrically by ``Vinberg deforming'' the corresponding cohomology classes?
\end{question}

\end{document}